\documentclass[a4paper,12pt,leqno]{amsart}
\usepackage[english]{babel}
\usepackage{amsmath,amsthm,amssymb}
\usepackage{xcolor,graphicx,delarray}
\usepackage[T1]{fontenc}
\usepackage{floatflt,tikz,pgf,pgfplots}
\pgfplotsset{compat=1.18}
\usepackage[toc,page]{appendix}
\usepackage{comment}
\usepackage{cancel}
\numberwithin{equation}{section}
\usepackage{subcaption}
\newcommand{\la}{\lambda}
\newcommand{\al}{\alpha}

\newcommand{\be}{\begin{equation}}
\newcommand{\bea}{\begin{eqnarray}}
\newcommand{\beas}{\begin{eqnarray*}}
\newcommand{\ee}{\end{equation}}
\newcommand{\eea}{\end{eqnarray}}
\newcommand{\eeas}{\end{eqnarray*}}

\newcommand{\tr}{{\rm Tr}}

\newcommand{\va}[1]{\left| #1 \right|} 
\newcommand{\laminates}{\mathcal{L}(\mathbb{M}^{3\times 3})}
\newcommand{\measures}{\mathcal{M}(\mathbb{M}^{3\times 3})}

\newcommand{\matrici}{\mathbb{R}^{3 \times 3}}
\newcommand{\R}{\mathbb{R}}
\newcommand{\N}{\mathbb{N}}
\newcommand{\Sd}{\mathbb{S}^2}
\newcommand{\holder}[1]{{\left[ #1 \right]}_{C^{\alpha} (\bar{\Omega}) }}

\newcommand{\dist}{\rm dist}

\newcommand{\LL}{{\mathcal L}(S)}

\newcommand{\T}{\mathcal T}

\DeclareMathOperator{\spt}{spt}
\DeclareMathOperator{\rank}{rank}

\DeclareMathOperator{\Tr}{Tr}
\newtheorem{theorem}{Theorem}[section]
\newtheorem{Lemma}[theorem]{Lemma}
\newtheorem{proposition}[theorem]{Proposition}

\newtheorem{corollary}[theorem]{Corollary}
\newtheorem{definition}[theorem]{Definition}
\newtheorem{remark}[theorem]{Remark}

\newlength{\temparglen}

\newcommand{\pagenlarge}[1]{
\setlength{\temparglen}{#1\topmargin}
\addtolength{\textheight}{2\temparglen}
\addtolength{\topmargin}{-\temparglen}
\setlength{\temparglen}{#1\oddsidemargin}
\addtolength{\textwidth }{2\temparglen}
\addtolength{\oddsidemargin }{-\temparglen}
\addtolength{\evensidemargin }{-\temparglen}
}
\pagenlarge{1}
\title[ Differential inclusions and polycrystals] {Differential inclusions and polycrystals}

\author[N. Albin]
{N. Albin}
\address[Nathan Albin]{Department of Mathematics, Kansas State University, 138 Cardwell Hall, 1228 N.~17th Street, Manhattan, KS 66506, USA}
\email{albin@k-state.edu}
\author[V. Nesi]
{V. Nesi}
\address[Vincenzo Nesi]{Dipartimento di Matematica, Sapienza, Universit\`a di Roma, P.le A. Moro, 00100, Rome, Italy}
\email{vincenzo.nesi@uniroma1.it}

\author[M. Palombaro]
{M. Palombaro}
\address[Mariapia Palombaro]{DISIM, Universit\`a dell'Aquila, Via Vetoio, 67100 L'Aquila, Italy}
\email{mariapia.palombaro@univaq.it}

\begin{document}
\begin{abstract}
We study the differential inclusion $Du\in K$, where $K$ is an unbounded and rotationally invariant subset of the real symmetric $3\times 3$ matrices.
We exhibit a subset of all possible average fields. 
%which is always contained in the convex-hull $K^c(S)$.
%of the set of points obtained by the triple $S$ and its permutations.approximate solutions to a differential inclusion.
%The datum is a set of three positive numbers identified with a positive definite diagonal matrix $S$.The aim is tiWe are given a system of pdes in dimension three.   
%The set of all possible average fields is called  %The set $K^{\rm qc}(S)$ is known to be stable under lamination. 
%We find a set $\LL$ of attainable average fields via a concise geometrical characterization in the  $3\times 3$ symmetric matrices space.
%strictly contained in the convex hull.
%and stable under lamination.
The corresponding microgeometries are laminates of infinite rank. 
%that may be of independent interest. 
The problem originated in the search for the effective conductivity of polycrystalline composites. In the latter context, our result is an improvement of the previously known bounds established by Nesi $\&$ Milton \cite{NM}, hence proving the optimality of a new full-measure class of microgeometries.

\end{abstract}

\maketitle
%\tableofcontents

\thispagestyle{empty}

\vspace{0.5cm}
\begin{flushright}
\itshape 
\tiny
Our paper is dedicated to the memory of Marco Avellaneda,\\
who made a profound contribution to the polycrystal problem, \\
who donated his brilliant insight to many branches of mathematics, \\
who defined himself  ``immigrant'' before mathematician, \\
in times when you needed the courage to do so.\\
A great man who stood firmly against discrimination. \\
In math and elsewhere.
\end{flushright}
\vspace{0.5cm}

\section{Introduction}\label{sec:introduction}
\noindent 
We provide solutions to a differential inclusion arising in the context of bounding the effective conductivity of polycrystalline composites  \cite{ACLM}. 
%\par
Consider the $3\times 3$ diagonal matrix 
\begin{equation}\label{S}
S=\left(
\begin{array}{ccc}
s_1&0&0\\
0&s_2&0\\
0&0&s_3
\end{array}
\right)
\end{equation}
subject to the constraints
\begin{equation}\label{cS}
0<s_1 < s_2< s_3,\quad s_1+s_2+s_3=1.
\end{equation}

In this paper, we shall visualize the unit-trace matrices as points in the plane using the projection shown in Figure~\ref{fig:2d-projection}.

%%%%%%%%%%%%%%%%%%
%%%%%%%%%%%%%%%%%%
\begin{figure}
    \centering
    \begin{subfigure}{0.33\textwidth}
    \includegraphics[width=\textwidth]{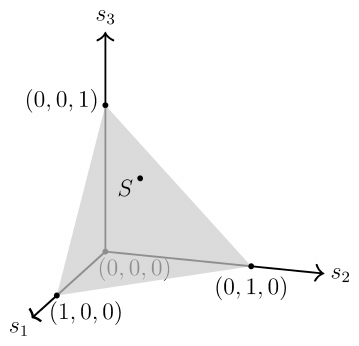}
    \caption{}\label{fig:2d-projection-a}
    \end{subfigure}%
    \begin{subfigure}{0.33\textwidth}
    \includegraphics[width=\textwidth]{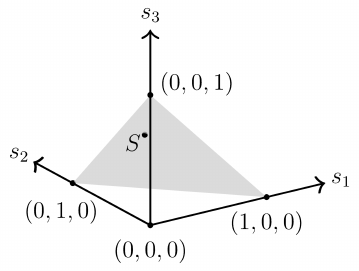}
    \caption{}\label{fig:2d-projection-b}
    \end{subfigure}
    \begin{subfigure}{0.33\textwidth}
    \includegraphics[width=\textwidth]{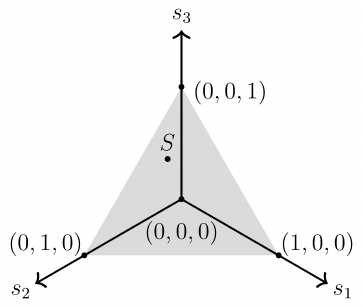}
    \caption{}\label{fig:2d-projection-c}
    \end{subfigure}%
    \caption{The projection of the unit-trace plane onto $\mathbb{R}^2$. Every $3\times 3$ real symmetric matrix, $S$, is visualized as a point in $\mathbb{R}^3$ corresponding to its eigenvalues, $(s_1,s_2,s_3)$ (up to permutation). The shaded triangle is the portion of the unit-trace plane formed by the convex hull of the canonical basis vectors. Figure~\ref{fig:2d-projection-a} shows the plane and a matrix, $S$, on the plane viewed from the $(+,+,+)$-octant. Figure~\ref{fig:2d-projection-b} shows the same thing but viewed from the $(-,-,+)$-octant. Figure~\ref{fig:2d-projection-c} shows a third view of the same configuration giving rise to the 2D projection. Here, the viewer sees the plane orthogonally from behind the origin. In this figure, $0<s_1<s_2<s_3$.}
    \label{fig:2d-projection}
\end{figure}
%%%%%%%%%%%%%%%%%%%%
%%%%%%%%%%%%%%%%%%%%

The assumption of strict inequalities in \eqref{cS} has the physical meaning 
that the polycrystal comprises a crystal that is not uniaxial. 
We denote by $\mathbb M^{3\times 3}$ the set of real 
$3\times 3$ matrices and by $\mathbb M^{3\times 3}_{\rm sym}$  its subset of symmetric matrices.
We define the 
set $K(S)\subset \mathbb M^{3\times 3}_{\rm sym}$ as follows:
\begin{equation}\label{setK}
K(S) :=\{
\lambda R^t S R: \lambda\in\R, \, R\in SO(3)
\}.
\end{equation}
Notice that $K(S)$ is unbounded.
Set  $C=[0,1]^3$ and denote by $W_{C}^{1,2}(\mathbb R^3;\mathbb R^3)$  the space of vector fields in 
$W^{1,2}_{\rm loc}(\mathbb R^3;\mathbb R^3)$ that are $C$-periodic.  
We look for  $A\in\ \mathbb M^{3\times 3}_{\rm sym}$ such that the following differential inclusion admits solutions
\begin{equation}\label{diffin}
 \nabla u\in K(S) \, \text{ a.e.}, \quad u-Ax \in W_{C}^{1,2}(\R^3; \R^3).
\end{equation}
Solutions are understood in the approximate sense of the existence of sequences $\{u_j\}\subset W_{C,A}^{1,2}\equiv W_{C}^{1,2}(\R^3; \R^3)+ A x$, 
that are bounded in $W^{1,2}_{\text{loc}}(\R^3;\R^3)$ and such that $\{\nabla u_j\}$ is 
$L^2_{\text{loc}}$-equi-integrable and
\begin{equation*}
{\dist} (\nabla u_j, K(S))\to 0 \text{ locally in measure}. 
\end{equation*}
%The set of all such $A$'s is known as the quasiconvex hull of $K(S)$ and is denoted by $K^{qc}(S)$ (see, e.g., \cite{Mnotes} for a  general introduction to the subject).
We denote the set of all such $A$'s as $K^{app}(S)$. In the case when a set $E\subset \mathbb M^{3\times 3}$ is bounded, 
it is well-known that $E^{app}$ coincides with the quasiconvex hull of $E$ (see, e.g., \cite{Mnotes} for a general introduction to the subject). For unbounded sets, this may not be the case and different notions of quasi-convex hulls are to be considered (see, e.g.,\cite{Zhang},  \cite{Yan}).  
In the present setting $K^{app}(S)$ is related to the so-called 2-quasiconvex hull of $K(S)$ as defined in \cite{Yan}, with $K^{app}(S)$ being a subset if it.  
We point out that a slightly different definition of 2-quasiconvex hull, less restrictive than the one in \cite{Yan}, can be found in \cite{Zhang}.  
By definition of $K(S)$,
$K^{app}(S)$ is invariant under conjugation by any rotation, i.e., if $A\in K^{app}(S)$ then $ R^t A R \in  K^{app}(S)$ for each $R\in SO(3)$. Therefore, it suffices to characterize the eigenvalues of the 
elements of $K^{app}(S)$, which can thus be identified with a subset of $\R^3$. 
In fact, because of its original physical motivation (see Section \ref{origin}), we are interested in those elements of $K^{app}(S)$ whose eigenvalues lie in the interval  $[s_1,s_3]$.  
Moreover, since $K^{app}(S)$ is a cone, we may then focus on the 
specific section Tr$A=1$, which leads to the following definition 
\begin{equation}\label{defKstar}
 K^*(S):= \left\{A\in K^{app}(S)\!: A= %\small{
 \left(
\begin{array}{ccc}
a_1&0&0\\
0&a_2&0\\
0&0&a_3
\end{array}
\right)\hspace{-1mm},
\, a_i \in [s_1,s_3], \, 
\sum_{i=1}^3a_i=1
\right\}.
\end{equation}
The set $K^*(S)$ can thus be identified with points in $\R^3$ lying in the hexagon of vertices at $S$ and all of its permutations.
The exact characterization of $K^*(S)$ is currently an open problem. Partial results were given in \cite{ACLM}, \cite{NM} and \cite{Nesi}.  In the present paper, we improve upon those results by exhibiting a set of attainable fields, which we denote $\LL$, strictly containing the previously known one established by Nesi $\&$ Milton \cite{NM} (see Proposition~\ref{MN_prop}).
%and it
%in addition to being optimal because it saturates a polyconvex bound, 
%enjoys the stability under lamination property, defined in Section \ref{sec-t2-bis} and proved in Sections \ref{More} and \ref{sec-stability}. 
Perhaps more importantly, we introduce micro-geometries displaying features that appear to be new. The mathematical analysis follows a scheme that is 
 well understood and known by various names, including the infinite-rank lamination scheme 
(\cite{ts}, \cite{Sche}, \cite{AH}, \cite{ct}, \cite{Albin}). 
 Our new results apply to matrices $S$ with distinct eigenvalues. In the uniaxial case, when $S$ has an eigenvalue of double multiplicity, they provide no improvements upon the previously known results.
 Our paper achieves an efficient and relatively quick scheme.
First, we introduce a class of  putative ``seed materials'' in the language of Nesi $\&$ Milton \cite{NM}, i.e., materials that can be shown to belong to $K^*(S)$ through an infinite-rank lamination in the spirit of the so-called Tartar's square \cite{ts}. 
We describe this class, which we denote by $\mathcal T^2(S)$ and show in Figure~\ref{fig:T2-locations}, by requiring the existence of specific rank-one connections, see \eqref{main}. 
 The latter implies, via Theorem \ref{laminateseq} and Corollary \ref{quasi-final}, the existence of a trajectory in $K^*(S)$ that starts in $S$, passes through $T$ and lands at a point that shares the same eigenvalues of $T$.   
 The reader may visualize the curve traced on $K^*(S)$ with the help of Figure \ref{fig:loops-and-masks}. 
 Note that such trajectories cross the uniaxial lines, i.e., the set of points where two eigenvalues coincide.
 \par
The second step of our scheme consists in constructing trajectories connecting $S$ with each uniaxial point in $\mathcal T^2(S)$ (the points $U_\alpha$ and $U_\beta$ in Figure~\ref{fig:sextant-annotated}). 
The associated rank-one connections in matrix space can be shown to exist and be unique, see Proposition \ref{prop_opt_curve}. 
The resulting trajectories in $K^*(S)$ defines the bounday of $\LL$.
The interior of $\LL$ is recovered by trajectories joining points on the boundary of  $\LL$.
Our scheme provides a quick selection of optimal microgeometries in a problem where one has an infinite choice of parameters, such as directions of laminations and rotations of the basic crystal.
It is reminiscent of work done in \cite{Milton2} for two-dimensional elasticity.
\par
In a forthcoming paper we will prove that $\LL$ enjoys the so-called stability under lamination (see Remark \ref{stability}). 
 We conclude by mentioning two natural open problems which could be the object of future work. First, the  lack of outer bounds for $K^*(S)$, even in the uniaxial case and even for exact solutions of \eqref{diffin}. Second,
 we do not know whether $\LL$ coincides with the 
rank-one convexification of $K(S)$ or is a subset of it
 %the understanding of the relation between $\LL$ and the 
  (see, e.g., \cite{Mnotes}, \cite{Yan}, \cite{Zhang} for the notion of generalised convex hulls).

%%%%%%%%%%%%%%%%%%%%%%%%%%%%%%%%%%%
\section{Physical motivation: the polycrystal problem}\label{origin}
The differential inclusion, \eqref{diffin}, has a physical interpretation in the context of conductivity, which we summarize briefly. An outstanding review can be found in 
\cite[p. 19]{masterpiece}. In this interpretation, one considers equations describing a large class of physical problems in which one has a simply connected region $\Omega\subset \mathbb R^3$ occupied by a physical material. The equations give a linear (constitutive) relation of the form
\begin{equation}\label{Maxwell}
    j(x)= \sigma(x) e(x),\quad x\in \Omega.
\end{equation}
The field $e(x)$ is curl-free and the field $j(x)$ is divergence-free. The law \eqref{Maxwell} is an instance of the stationary Maxwell equations and describes many different physical problems such as electrical conductivity, dielectric materials, magnetism, thermal conduction, diffusion, flow in porous media, and antiplane shear. For electrical conductivity, $e(x)=-\nabla u$, with $u$, $e(x)$, $j(x)$ denoting electrical potential, electrical field and electrical current, respectively. The matrix $\sigma(x)$ is called the electrical conductivity, and \eqref{Maxwell} is a way of stating Ohm's law: there is a (locally) linear relation between $e$ and $j$. The conducting body is called homogeneous if $\sigma$ does not depend on $x\in \Omega$, and isotropic if $\sigma(x)$ is invariant under the action $\sigma\to R^t \sigma R$, for any $R\in SO(3)$. For classical materials, where $\sigma(x)$ is a symmetric matrix for every $x$, isotropy is equivalent to assuming that $\sigma$ is proportional to the identity matrix. In general, the eigenvalues of the conductivity matrix are called principal conductivities.  If at least two principal conductivities of $\sigma$ are distinct, then the material is called a crystal. 
This simply means that it conducts, like most crystals in nature, with different efficiency along different directions. By the linearity of Ohm's law, knowledge of the materials conductivity properties in three linearly independent directions is sufficient to determine $\sigma$. Special attention (see \cite{Sch}), has been given to the so-called {\it uniaxial} crystals, i.e., those for which two principal conductivties coincide, or, in geometrical terms, that exhibit cylindrical symmetry. In general, materials are not homogeneous. In inhomogeneous materials, the conductivity depends upon $x\in \Omega$. A very simple example is a material in which $\Omega$ is partitioned into $N$ regions $\Omega_j$ and in each of them one has a different matrix $\sigma$. In particular,
given a diagonal matrix 
\begin{equation}\label{Sigma}\Sigma=\left(
\begin{array}{ccc}
\sigma_1&0&0\\
0&\sigma_2&0\\
0&0&\sigma_3
\end{array}\right),\quad 0<\sigma_3< \sigma_2  <\sigma_1,
\end{equation} 
one may choose 
\begin{equation}\label{sigmazero}
\sigma(x)=\sum_{j=1}^NR_j^t(x) \Sigma R_j(x),\quad R_j\in SO(3).
\end{equation}
Such a conductivity describes a material which is anisotropic, and not homogeneous. Since the eigenvalues of $\sigma$ do not change, one speaks of a ``polycrystal'', because there are $N$ crystals, ``made of the single crystal'' with homogeneous conductivity $\Sigma$. The physical material is the same in every region; only its local orientation varies. We say that this represents a mixture of a given crystal with itself using $N$ orientations. In nature, $N$ can be extremely large. Mathematically, it makes sense to consider
$x\to R(x)\in SO(3)$ to be a generic measurable field, i.e.
\begin{equation}\label{sigma}
\sigma(x)=R^t(x) \Sigma R(x).
\end{equation}
In the theory of composite materials, one addresses the problem of determining some macroscopic or overall behaviour. We recall some well-known facts (see  \cite[Chapter 1]{masterpiece}).
One introduces the
 so-called ``effective'' or homogenized conductivity, $\sigma^*$, defined as follows:
\begin{equation}\label{effcondzero}
\sigma^* a \cdot a:=\inf_{u\in W_{C,a}^{1,2}}
\int_{C} \sigma(x)\nabla u(x) \cdot \nabla u(x) \,dx,\quad \forall a \in \mathbb R^{3},
\end{equation} 
where $W_{C,a}^{1,2}\equiv W_{C}^{1,2}(\R^3; \R)+ a\cdot  x$ is the scalar version of the space $W^{1,2}_{C,A}$ introduced in Section~\ref{sec:introduction}. 
It is well-known
that $\sigma^*$ is a $3\times 3$ symmetric matrix, and depends neither on $x$ nor on $a$, therefore, it can be determined by testing the variational principle \eqref{effcondzero} on three independent vectors $a^i$. By the linearity of the underlying Euler-Lagrange equations, we may sum three different contributions relative to three vectors $a^i$ and potentials $u^i$ and write 
\begin{equation*}
\sum_{i=1}^3\sigma^* a^i \cdot a^i=\sum_{i=1}^3\inf_{u^i\in W_{C,a^i}^{1,2}}
\int_{C} \sigma(x)\nabla u^i(x) \cdot \nabla u^i(x) \,dx.
\end{equation*}
This implicitly defines a $3\times 3$ matrix $A$ and vector potential $u=(u^1,u^2,u^3)$ such that 
\begin{equation}\label{effcond}
{\rm Tr}(A \sigma^* A^t):=\inf_{u\in W_{C,A}^{1,2}}
\int_{C} {\rm Tr}(Du(x) \sigma(x) Du^t(x)) \,dx,\quad \forall A \in \mathbb M^{3\times 3},
\end{equation} 
where, again, $W_{C,A}^{1,2}\equiv W_{C}^{1,2}(\R^3; \R^3)+ A x$. 
From a physical perspective, $\sigma^*$ describes the overall response of the composite material to an imposed external average field $\int_C Du(x)dx = A$.  Since the composite material involves fine-scale structure, $\sigma$ and $Du$ exhibit rapid oscillations inside of the periodic cell, $C$. The homogenized conductivity tensor averages out these rapid oscillations by directly relating the \emph{average} behaviors of the current and electric field across the entire periodic cell. Namely,
\begin{equation*}
\int_C \sigma Du^t\;dx = \sigma^*\int_C Du^t\;dx=\sigma^*A^t.
\end{equation*}

The set of all possible $\sigma^*$ that may arise while $R$ varies in $L^\infty(C;SO(3))$ is sometimes called the $G$-closure, denoted by $G(\Sigma)$. 
\begin{comment}
The name is related to the concept of $G$-convergence, introduced by Spagnolo \cite{spagnolo} who was certainly influenced by De Giorgi's ideas.  Later Murat and Tartar \cite{MT}, generalized the $G$-convergence to cover the non-variational case of $\sigma$'s  which are elliptic but not symmetric,
to $H$-convergence.
Well-known references on the topic are also Bensoussan, Lions $\&$ Papanicolaou \cite{blp} and Jikov, Kozlov $\&$ Olienik \cite{jko}. 
\end{comment}
%
The $G$-closure, in the case under study, is a set of symmetric, positive definite matrices that is rotationally invariant, in the sense that if a diagonal matrix $\sigma^*\in G(\Sigma)$, then $R^t \sigma^* R\in G(\Sigma)$
 for any constant matrix $R\in SO(3)$. Therefore, it suffices to study the range of the eigenvalues $\sigma_i^*$ of effective conductivities. 
The first major advance in the polycrystal problem was obtained by Avellaneda et al. \cite{ACLM}, which established several optimal bounds and obtained some important partial results in terms of optimal microgeometries. 
The next result summarizes the main bounds found in \cite[Section 2]{ACLM}.
\begin{theorem}\label{avetal}
Let $\sigma$ be given as in \eqref{sigma}. Then $\sigma^*$ satisfies 
\begin{align} 
\label{cube-trace}
& \sigma_3 \leq \sigma^*_i \leq \sigma_1,\quad
{\rm Tr}\,\sigma^*\leq  {\rm Tr}\sigma,\\
\label{optimalitystar} 
& \det \sigma^*-\theta^2\,{\rm Tr} \sigma^*-2 \theta^3\geq 0,
\end{align}
where $\theta$ is  the least positive solution of \,
$\det \sigma-\theta^2\,{\rm Tr} \sigma-2 \theta^3=0$.
\end{theorem}
The left-most bound in \eqref{cube-trace} follows immediately by the ellipticity of the matrix $\sigma$. The right-most is found using the affine test field $u(x) = A x$ in \eqref{effcond}.  The bound \eqref{optimalitystar}, instead, is one of the first instances of a rather elegant polyconvexification argument, (see \cite[Section 2]{ACLM}). 
%Precicely the authors find that $\theta$ si the least positive number such that the quadratic form
%$${\rm} Tr (A^T \sigma(x) A)- 2 \theta i_2(A)\geq 0, $$
%for all $3\times 3$ matrices $A$. In fact, for such a choice of $\theta$, the kernel of the associated linear operator is non trivial. Optimality of a micrgeometry is then equivalent to showing that one can find a field 

The problem we pose is the attainability of the bound \eqref{optimalitystar}, which corresponds to establishing which $\sigma^*$ (identified with its eigenvalues) 
actually
lie on the convex surface determined by \eqref{optimalitystar},  which represents a portion of the boundary of $G(\Sigma)$.
The optimality of \eqref{optimalitystar} was established in \cite{ACLM} only under the severe condition that $\sigma$ be uniaxial, i.e., exactly two eigenvalues coincide. The construction uses a famous example by Schulgasser \cite{Sch} and, in the nowadays language would be called an ``exact solution'', in particular the microgeometry allows for a solution of \eqref {effcond} with $A=I$ and $u(x)= x\lvert x\rvert^{\alpha}$, where $\alpha$ is an appropriate exponent depending on the ratio of the two distinct eigenvalues of $\sigma$. 
A further advancement on the optimality of the lower bound \eqref{optimalitystar} was 
made by  Nesi \& Milton \cite{NM}, who recast the problem as a differential inclusion as 
clarified by the following lemma.
\begin{Lemma}\label{MiltonNesiThm}
Let 
\begin{equation}\label{defS}
S:=
\left(
\begin{array}{ccc}
\frac{\theta}{\theta +\sigma_1}&0&0\\
0&\frac{\theta}{\theta +\sigma_2}&0\\
0&0&\frac{\theta}{\theta +\sigma_3}
\end{array}
\right),
\end{equation}
with $\theta$ as in Theorem \ref{avetal}. 
If $S^*\in K^*(S)$, with $K^*(S)$ defined by \eqref{defKstar}, then $\sigma^*:= \theta((S^*)^{-1}-I)$ belongs to $G(\Sigma)$ and saturates the bound  
\eqref{optimalitystar}.
\end{Lemma}
Lemma \ref{MiltonNesiThm} implies that the bound \eqref{optimalitystar} is attained if there exist (approximate) solutions to the differential inclusion \eqref{diffin} with $S$ given by \eqref{defS}, which, by definition of $\theta$, satisfies $\tr S=1$.
%This establishes the connection between the result in the present paper and the polycrystal problem. 
%From the mathematical point of view the behavior is similar, but not identical to the famous Serrin example \cite{serrin}. 
From \eqref{defS} and the relationship between $S^*$ and $\sigma^*$ one can see that 
the eigenvalues of $S^*$ must lie in the interval $[s_1,s_3]$, 
which leads to the definition \eqref{defKstar} of $K^*(S)$.
The main contribution of \cite{NM} was to consider arbitrary $\sigma$'s, in particular non uniaxial ones, and prove that a large part of the surface defined by \eqref{optimalitystar} is actually attained using an infinite lamination procedure (see Section \ref{vecchio}).
 The Milton-Nesi construction resembles the well-known Tartar's square \cite{ts} 
 (see also \cite{Sche}, \cite{AH} and \cite{ct}, which use a very similar construction)
 %(see also the work on bi-martingales by Aumann and Hart \cite{AH} and the one by Casadio Tarabusi \cite{ct}, which use a very similar construction), 
 and permits to find a set $Z$ of three $3\times 3$ matrices such that they are not rank-two connected, but for which an approximate solution to the differential inclusion $B\in Z, {\rm Div}\, B=0$ exists (see also \cite{GN}, \cite{PP}, \cite{PS} and \cite{Ang}).
 %It is also interesting to note that the same idea was presented much earlier in the (unpublished) Dissertation by Scheffer, \cite{Sche}.
 In fact, the problem of the optimality of the bound 
 \eqref{optimalitystar}, as well as other bounds for effective conductivities, can be equivalently rewritten as a differential inclusion of Div-free type. This route has not been exhaustively pursued yet and may be the object of future work.
%%%%%%%%%%%%%%%%%%%%%%%%%%%%%%%%%%%
\section{Rank-one connections: the set 
$\mathcal T^{1}(S)$}
\label{sec-T1}
\noindent
We look for the set of symmetric matrices with unit trace that are rank-one connected to a scalar multiple of $S$.
In what follows $\Sd$ denotes the set of unit vectors in $\R^3$.
We will need the following definition.
\begin{definition}\label{MT}
\noindent
Let $S$ be as in \eqref{S}, \eqref{cS}.  The set 
${\mathcal T^{1}(S)}$ consists of diagonal matrices $T$ with eigenvalues $t_i$ satisfying 
\begin{equation}\label{ordering}
s_1\leq t_1\leq t_2\leq t_3\leq s_3,  %\hbox{\quad\hidden{E. maybe not}}
\end{equation}
%which in addition satisfies:
\begin{equation}\label{oldbp}
\begin{array}{cc}
\exists \,R\in SO(3), n\in\Sd, \la \in \mathbb R :
&
R^tTR = \la\, S +(1-\la) n\otimes n. 
\end{array}\end{equation}
\end{definition}
Note that, by \eqref{oldbp}, $t_1+t_2+t_3=1$. To achieve a representation of the set $\mathcal T^{1}(S)$,
 we define the following sets of numbers and intervals.
\begin{definition}\label{A(T,S)}
Assume that the pair $(T,S)$ satisfy \eqref{S}, \eqref{cS} and \eqref{ordering}. Set
\begin{align*}
& \al_- (T,S) :=\max\left(\frac{t_1}{s_2},\frac{t_2}{s_3}\right), 
& \al_+ (T,S):=\min\left(\frac{t_2}{s_2},\frac{t_3}{s_3}\right) ,\\
&\beta_-(T,S):= \max \left(\frac{t_1}{s_1},\frac{t_2}{s_2}\right),
&\beta_+(T,S):=\min \left(\frac{t_2}{s_1},\frac{t_3}{s_2}\right), \\
& A_\alpha(T,S):=[\al_- (T,S), \al_+ (T,S)],
&A_\beta(T,S):=[\beta_-(T,S), \beta_+(T,S)],\\
&A(T,S) :=A_\alpha(T,S)\cup A_\beta(T,S). 
\end{align*}
The interior and the boundary of $A(T,S)$ are denoted by $A^{\circ}(T,S)$ and $\partial A(T,S)$ respectively, and we adopt the convention $[a,b] =\emptyset$, if $a>b$.
\end{definition}
\begin{remark}\label{rem100}
From the definition it follows that  $\alpha_+(T,S)\leq 1\leq \beta_-(T,S)$ and that  $\alpha_+(T,S)=\beta_-(T,S)=1$ if and only if $t_i=s_i$ for each $i=1,2,3$.  
Hence, if $\{s_1,s_2,s_3\}\neq \{t_1,t_2,t_3\}$,  the set $A(T,S)$ is always the union of two disjoint bounded intervals. 
\end{remark}
\noindent
The next algebraic lemma clarifies when condition \eqref{oldbp} holds. Its proof is postponed to Section \ref{dimoLemma}.
%Recall Definition \ref{MT} for $\mathcal T^{1}(S)$.
%
\begin{Lemma}\label{LemmaF}
Assume that the pair $(T,S)$ satisfy \eqref{S}, \eqref{cS} and \eqref{ordering}.
If $T\in \mathcal T^{1}(S)$, then $\la \in  A(T,S)$. Conversely, if $A(T,S) \neq \emptyset$, then $T\in \mathcal T^{1}(S)$.
Furthermore, if $A(T,S) \neq \emptyset$, then for each
$\la\in  A(T,S)\setminus{\{1\}}$, the vector $n=(n_1,n_2,n_3)$ that satisfies \eqref{oldbp} is 
determined, not uniquely, by the following equations
\begin{equation}\label{ni}
\begin{array}{cc}
\displaystyle{n_1^2 = \frac{(t_1 -\la s_1)(t_2 -\la s_1)(t_3 -\la s_1)}{\la^2(1-\la)
(s_2 - s_1)(s_3 - s_1)} := n_1^2(T,S,\la)},
\\\\
\displaystyle{n_2^2 = \frac{(t_1 -\la s_2)(t_2 -\la s_2)(t_3 -\la s_2)}{\la^2(1-\la)
(s_3 - s_2)(s_1 - s_2)} := n_2^2(T,S,\la)},
\\\\
\displaystyle{n_3^2 = \frac{(t_1 -\la s_3)(t_2 -\la s_3)(t_3 -\la s_3)}{\la^2(1-\la)
(s_1 - s_3)(s_2 - s_3)} := n_3^2(T,S,\la)}.
\end{array}
\end{equation}
\end{Lemma}

%%%%%%%%%%%%%%%%%%%%%%%%%%%%%%%%%%%%%%%%% 
\begin{remark} \label{noninteressante}
Note that
 \begin{equation}\label{casoT=S}
 A(S,S)=\left[\alpha_-(S,S),\frac{1}{\alpha_-(S,S)}\right] =
 \left[\max\left(\frac{s_1}{s_2}, \frac{s_2}{s_3}\right), \min\left(\frac{s_2}{s_1}, \frac{s_3}{s_2}\right)\right]
 \neq \emptyset,
\end{equation}
and therefore $S\in \mathcal T^{1}(S)$. 

As already observed in Remark \ref{rem100}, we have $1\in A(T,S)$ if and only if $t_i=s_i$ for each $i=1,2,3$.  
The case $\la=1$ is not interesting since there is no actual rank-one connection and the relation \eqref{oldbp} is trivially satisfied by any $n$. 
Moreover if $S$ is uniaxial, it can be easily checked that   $A(S,S)=\{1\}$.   
\end{remark} 
\section{The set $\mathcal T^2(S)$}
\label{sec-t2}
We now define a subset of ${\mathcal T}^1(S)$ which we denote by $\mathcal T^2(S)$. We will prove later that 
$\mathcal T^2(S)\subset K^*(S)$ (see Corollary \ref{quasi-final}).
 Define the function $F:[s_1,s_3]^2 \to \R$  as
\begin{equation}\label{defF}
    F(x,y):=\frac{s_1\,s_3}{s_2}\cdot \frac{\,x\,y}{x^2+x\,y+y^2-s_2\,(x+y)}.
\end{equation}
\begin{definition}\label{defT22} 
Let $S$ be as in \eqref{S}, \eqref{cS}.
The set  $\T^2(S)$ consists of all  $T$ that, in addition to \eqref{ordering}, satisfy either
    \begin{equation}\label{condnec1}
\left\{\begin{array}{ll}
(\alpha_-(T,S),\alpha_+(T,S))=\displaystyle{\left(\frac{t_1}{s_2},\frac{t_2}{s_2}\right)}\\\\
t_3=F(t_1,t_2)
\end{array}
\right.
\end{equation}
or
\begin{equation}\label{condnec2}
\left\{\begin{array}{ll}
(\beta_-(T,S),\beta_+(T,S))=\displaystyle{\left(\frac{t_2}{s_2},\frac{t_3}{s_2}\right)}\\\\
t_1=F(t_3,t_2)
\end{array}\right.
\end{equation}
where $F$ is given by \eqref{defF}.
\end{definition}
The curves defined by \eqref{condnec1} and \eqref{condnec2} can be 
visualised in the unit-trace plane with the help of Figure \ref{fig:T2-locations}. 
\begin{remark}
Note that
$$(\alpha_-(T,S),\alpha_+(T,S))=\displaystyle{\left(\frac{t_1}{s_2},\frac{t_2}{s_2}\right)}
\Longleftrightarrow
\displaystyle{ \frac{t_2}{t_3}\leq \frac{s_2}{s_3}\leq  \frac{t_1}{t_2}},
$$
$$
(\beta_-(T,S),\beta_+(T,S))=\displaystyle{\left(\frac{t_2}{s_2},\frac{t_3}{s_2}\right)}
\Longleftrightarrow
 \frac{t_1}{t_2}\leq \frac{s_1}{s_2}\leq \frac{t_2}{t_3}.
$$
\end{remark}
\begin{remark}
One has that $\mathcal T^2(S)\subset \mathcal T^1(S)$. This follows from Proposition 
\ref{prop5.2} if $T$ is not uniaxial, and from (ii) of Proposition \ref{prop_opt_curve} if $T$ is uniaxial.
\end{remark}
\begin{proposition}\label{prop5.2}
Let $T\in  \mathcal T^2(S)$ have distinct eigenvalues and let
\begin{equation}\label{scelta}
(\lambda_1,\lambda_2)=
\left\{\begin{array}{ll}
\displaystyle{\left(\frac{t_1}{s_2},\frac{t_2}{s_2}\right)}\quad 
\textit{if \eqref{condnec1} holds,}\\\\
\displaystyle{\left(\frac{t_3}{s_2},\frac{t_2}{s_2}\right)} \quad 
\textit{if \eqref{condnec2} holds.}
\end{array}
\right.
\end{equation}
Then there exist  $R_1, R_2 \in SO(3),\, n\in \Sd$ such that 
\begin{equation}\label{main}
R_1^t TR_1 = \la_1 S +(1-\la_1)\, n\otimes n,\quad
R_2^t TR_2 = \la_2 S +(1-\la_2)\, n\otimes n.  
\end{equation}
\end{proposition}

The proof of Proposition \ref{prop5.2} is postponed to Section \ref{dimopropoT2}.
%%%%%%%%%%%%%%%%%%%%%
%%%%%%%%%%%%%%%%%%%%%%
%%%%%%%%%%%%%%%%%%%%%%%%
\subsection{Solution to the underlying differential inclusion}\label{sec-di}
\label{sec:differential-inclusion}
\noindent
The next  two results show that if $T\in \T^2(S)$, then all points that lie on a specific trajectory starting from $S$ and ending 
at $T$ on the unit trace plane actually belong to $ K^*(S)$.
This is proved by exhibiting an infinite rank laminate that generically uses an infinite set of rotations $R\in SO(3)$, and thus infinitely many distinct rank-one directions of lamination. 
Figure \ref{fig:loops-and-masks} shows  rank-one curves connecting $S$ and  matrices in $\T^2(S)$. By  Corollary \ref{quasi-final} such curves lie entirely in $K^*(S)$. 
In the sequel we will use the notions of laminate and splitting of a laminate, whose definitions are recalled in the Appendix \ref{CT} (see in particular Definition \ref{laminates}).
For $A,B\in \mathbb{M}^{3\times 3}$, we denote by $(A,B)$ and $[A,B]$
the open and closed segment connecting $A$ and $B$, respectively.

\begin{theorem}\label{laminateseq}
Let $T\in {\mathcal T}^2(S)$ have distinct eigenvalues and let $\lambda_1, R_1$ be given by 
Proposition \ref{prop5.2}. 
There exist $C>0$,  $\bar r>2$ such that for each $A\in (\lambda_1 S, R_1^tTR_1)$, there exists a sequence of laminates of finite 
order $\nu_k\in  \laminates$ satisfying
\begin{enumerate}
\item[(i)] $\bar\nu_k=A$ \  $\forall \ k\in\N$; \smallskip
\item[(ii)] $\nu_k(\mathbb{M}^{3\times 3}\setminus K(S))\to 0$ \ $k\to +\infty$; \smallskip
\item[(iii)] $\int_{\mathbb{M}^{3\times 3}} |F|^r\  d\nu_k(F) < C, \quad \, \forall \, k\in \N,\quad  \forall\, r\in [1, \bar r)$. 
\end{enumerate}
\end{theorem}
%%%%%%%%%%%%%%%%%%%%
\begin{proof}
We construct the sequence $\nu_k$ by successive splitting.
By Proposition \ref{prop5.2} we have
 \begin{equation}\label{R1R2}
 \left\{\begin{array}{lcc}
 R_1^tTR_1=\lambda_1\,S+(1-\lambda_1)\,n\otimes n 
 \\\\
 R_2^tTR_2=\lambda_2\,S+(1-\lambda_2)\,n\otimes n.
 \end{array}
\right.
 \end{equation}
Since $A\in (\lambda_1 S, R_1^tTR_1)$, there exists 
$p\in (0,1)$ such that $A= p\lambda_1 S + (1-p) R_1^tT R_1$.
We define the first laminate of the sequence as $\nu_1 := p\delta_{\lambda_1 S} + (1-p) \delta_{R_1^tTR_1} $. 
The next step is to replace $ \delta_{R_1^tTR_1} $ by the sum of two Dirac masses supported in rank-one connected matrices.
For this purpose let
\begin{equation}\label{q}
q:=\frac{\lambda_1(1-\lambda_2)}{\lambda_2(1-\lambda_1)}, \quad \lambda:=\frac{\lambda_2}{\lambda_1},
 \end{equation}
 and notice that,
since by \eqref{scelta} either $0<\lambda_1<\lambda_2<1$ or $1<\lambda_2<\lambda_1$, 
 one has that  $q\in (0,1)$ 
with $\lambda >1$ in the first case, and $\lambda<1$ in the second case.
 Moreover, by \eqref{R1R2}
 \begin{equation}\label{seconda}
\frac{1}{\lambda}R_2^tTR_2=\lambda_1 S+q(1-\lambda_1) n\otimes n.
 \end{equation}
 Let us define
 \begin{equation}\label{varie1}
 \left\{\begin{array}{llll}
S_0= \lambda_1 S, &  T_0=R_1^t T R_1, & Q=R_2^t  R_1,\\\\
S_1=\lambda Q^t  S_0 Q, &T_1=\lambda Q^t  T_0 Q,  &M=(1-q)S_1+q T_1.
\\
 \end{array}
 \right.
 \end{equation}
 \vspace{2mm}
 
 \noindent
 In view of \eqref{varie1} we can write $\nu_0 = p\delta_{S_0} + (1-p) \delta_{T_0} $. To perform the first splitting we 
 check  that $M=T_0$ and replace $\delta_{T_0}$ by  $(1-q)\delta_{S_1}+q \delta_{T_1}$. By the first equation in \eqref{R1R2}, we have
 $$
 T_0=S_0+(1-\lambda_1) n\otimes n.
 $$
Using \eqref{varie1}, we have
 \begin{equation}
\begin{array}{cccc}
 M=\lambda Q^t[(1-q) S_0 +q T_0]Q=\lambda Q^t[(1-q)S_0+ q S_0+q(1-\lambda_1)n\otimes n]Q\\\\
= \lambda Q^t[ S_0+q(1-\lambda_1)n\otimes n]Q=\lambda Q^t[ \lambda_1 S+q(1-\lambda_1)n\otimes n]Q.
 \end{array}
 \end{equation}
 Now we use \eqref{seconda} and the previous equation and get
 \begin{equation}
 Q M Q^t=
\lambda [ \lambda_1 S+q(1-\lambda_1)n\otimes n]=R_2^t T R_2.
 \end{equation}
 Therefore, we have
  \begin{equation}
 M=T_0 = R_1^t T R_1 \Longleftrightarrow Q T_0 Q^t=
R_2^t T R_2\Longleftrightarrow Q R_1^t T R_1 Q^t=R_2^t T R_2 \Longleftrightarrow 
R_2 Q R_1^t =I.
 \end{equation}
 The latter follows from the definition of $Q$.
 We can now define the second laminate as
 $$
 \nu_1 : =  p\delta_{S_0} + (1-p)[(1-q)\delta_{S_1}+q \delta_{T_1}].
 $$
 Notice that $\spt(\nu_1)\in K(S)\cup \{T_1 \}$.
 To iterate the above procedure, we introduce the following sequences:
\begin{equation}\label{formule-SkTk}
S_k:=\lambda^k (Q^k)^t S_0 Q^k, \quad T_k :=\lambda^k (Q^k)^t T_0 Q^k . 
\end{equation}
We note that for each $k$ the pair $(T_k,S_k)$ is rank-one connected. Indeed, 
   \begin{equation}\label{rango1con}
   T_k-S_k=\lambda^k (Q^k)^t (T_0-S_0) Q^k =\lambda^k (1-\lambda_1) (Q^k)^t n \otimes n Q^k.
 \end{equation}
 Moreover, for each $k$,
   \begin{equation}\label{induk}
T_k= (1-q) S_{k+1}+q T_{k+1}.
   \end{equation}
   We prove \eqref{induk} by induction. The case $k=0$ has been proved in the previous part. So assume 
     \begin{equation}
T_{k-1}= (1-q) S_{k}+q T_{k},
   \end{equation}
 and prove that
    \begin{equation}\label{toprove}
T_{k}= (1-q) S_{k+1}+q T_{k+1}.
   \end{equation}
   We start computing the right-hand side
       \begin{equation}
       \begin{array}{l}
 (1-q) S_{k+1}+q T_{k+1}=\lambda^{k+1}(Q^{k+1})^t[(1-q) S_0 +q T_0]Q^{k+1}
\\\\
=\lambda^{k+1} (Q^{k+1})^t[(1-q)S_0+ q S_0+q(1-\lambda_1)n\otimes n]Q^{k+1}=
\lambda^{k+1} (Q^{k+1})^t[ S_0+q(1-\lambda_1)n\otimes n]Q^{k+1}\\\\
=\lambda^{k+1} (Q^{k+1})^t[ \lambda_1 S+q(1-\lambda_1)n\otimes n]Q^{k+1}.
  \end{array} \end{equation}
  We now use \eqref{seconda} and get
  $$
  (1-q) S_{k+1}+q T_{k+1}=\lambda^{k+1} (Q^{k+1})^t\left[ \frac{1}{\lambda} R_2^t T R_2\right]Q^{k+1}.
  $$
  To prove \eqref{toprove}, we are left with proving that
  \begin{equation}
 T_k= \lambda^{k+1} (Q^{k+1})^t\left[ \frac{1}{\lambda} R_2^t T R_2\right]Q^{k+1},
  \end{equation}
  namely that
  \begin{equation}
\lambda^{k}(Q^k)^t T_0Q^k= \lambda^{k} (Q^{k+1})^t  R_2^t T R_2 Q^{k+1}.
  \end{equation}
  The latter is equivalent to
  \begin{equation}
T_0= Q^t R_2^t T R_2 Q\Longleftrightarrow R_1^t T R_1=Q^t R_2^t T R_2 Q\Longleftrightarrow R_1^t=Q^t R_2^t\Longleftrightarrow Q=R_2^t R_1 .
  \end{equation}
We can now define the sequence $\nu_k$ recursively. 
In order to obtain the laminate $\nu_{k+1}$ from $\nu_k$, we use  \eqref{rango1con} and \eqref{toprove} to replace $\delta_{T_{k}}$ in 
$\nu_k$ by $(1-q)\delta_{S_{k+1}}+q \delta_{T_{k+1}}$.
By construction each $\nu_k$ has barycenter $A$ and satisfies
\begin{equation*}
\spt(\nu_k) \subset K\cup \{ T_{k}\} ,\quad \nu_k(T_{k}) = (1-p)q^k \to 0 \quad k\to 0.
\end{equation*}
It remains to check that (iii) holds. Using \eqref{formule-SkTk} we explicitly compute 
 \begin{align*}
 \int_{\R^{3\times 3}} |F|^r\  d\nu_{k+1}(F) & = p|S_0|^r + (1-p) (1-q) \sum_{j=0}^k q^j |S_{j+1}|^r 
  + (1-p)q^{k+1} |T_{k+1}|^r 
 \\
 & = p |S_0|^r + (1-p) (1-q) |S_0|^r \lambda^r \sum_{j=0}^k (q\lambda^r)^j  
 + (1-p)(q\lambda^r)^{k+1}|T_0|^r.
\end{align*}
If $\lambda_1=\frac{t_3}{s_2}$ and $\lambda_2=\frac{t_2}{s_2}$, then one has that $\lambda\in (0,1)$ and $q\lambda^r\in (0,1)$ 
for each $r\geq 1$, which implies 
${\displaystyle \sup_k  \int_{\R^{3\times 3}} |F|^r\  d\nu_{k}(F) <C }$ for each $r\geq1$. 
If instead $\lambda_1=\frac{t_1}{s_2}<\lambda_2=\frac{t_2}{s_2}<1$, then 
$$
q\lambda^r = \frac{s_2 -t_2}{s_2 -t_1}\Big(\frac{t_2}{t_1}\Big)^{r-1}.
$$
Therefore if we set 
$$
\bar r:= 1+ \log_{\frac{t_2}{t_1}}\Big(\frac{s_2 -t_1}{s_2 -t_2}\Big),
$$ 
then, recalling that $t_1<t_2<t_3$ and observing that $(s_2-t_1)/(s_2-t_2) > t_2/t_1$, one has $\bar r >2$ and $q\lambda^r < 1$ for each $r < \bar r$, which yields (iii).
\end{proof}
\begin{remark}
The set of rank-one directions used in the proof of Theorem \ref{laminateseq}, namely the set $\{(Q^t)^k n : k\in \N\}$, generically contains infinitely many distinct elements. In the case when the rotation angle is rational, there exists $k\in \N$ such that  
$(Q^t)^k = Q^t$ and therefore the set of rank-one directions is finite.
\end{remark}
\begin{corollary}\label{quasi-final}
If $T\in {\mathcal T}^2(S)$, then for each $A\in[\lambda_1 S, R_1^tTR_1]$ there exists $R\in SO(3)$ such that the matrix ${\displaystyle \frac{1}{\Tr A}R^t A R}$ belongs to $K^*(S)$. 
\end{corollary}
\begin{proof}
If $T$ is not uniaxial, then the proof is a consequence of Theorem \ref{laminateseq} and Proposition \ref{gradienti}. 
Precisely, for each $k\in \N$, we define a map $u_k \in  W_{C,A}^{1,2}$ by using Proposition \ref{gradienti} applied to $\nu = \nu_k$ and 
$\delta = \delta_k$, where $\delta_k$ is any sequence satisfying
$$
\delta_k <\min\Big\{\frac{1}{k},\, \min_{A_i,A_j\in\,\spt \nu_k}  \frac 1 2 
|A_i - A_j|\Big\}.
$$ 
The sequence $\{u_k\}$ thus obtained is an approximate solution to \eqref{diffin}.
The case of uniaxial $T$ is recovered by a closure argument. 

\end{proof}
%%%%%%%%%%%%%%%%%%
%%%%%%%%%%%%%%%%%%
\begin{figure}\label{bellafigura}
    \centering
    \begin{subfigure}{0.5\textwidth}
    \includegraphics[width=\textwidth]{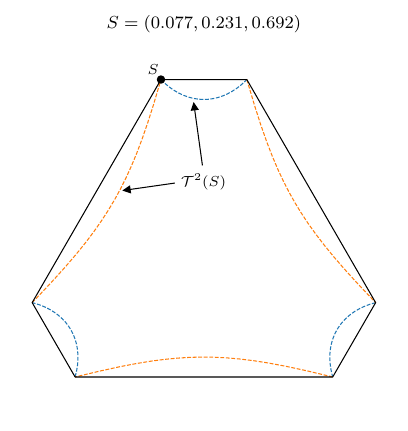}
    \caption{}
    \label{fig:T2-locations}
    \end{subfigure}%
    \begin{subfigure}{0.5\textwidth}
    \includegraphics[width=\textwidth]{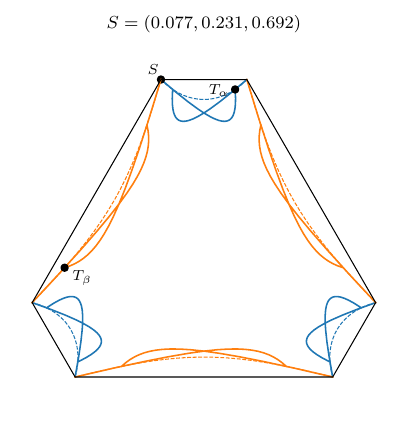}
    \caption{}
     \label{fig:loops-and-masks}
    \end{subfigure}
    \caption{The dashed curves in both figures show the set of $\mathcal{T}^{2}(S)$ fields defined in~\eqref{condnec1} and~\eqref{condnec2}. The solid curves in the right figure show rank-one curves that are formed by laminating $S$ with $\mathcal{T}^{2}(S)$ fields, $T_\alpha$ satisfying~\eqref{condnec1} and $T_\beta$ satisfying~\eqref{condnec2}.}
\end{figure}
%%%%%%%%%%%%%%%%%%%%
%%%%%%%%%%%%%%%%%%%%
\section{Analytic characterization of the set $\LL$}\label{sec-t2-bis}
\noindent  
The goal of the present section is to define the set $\LL$, which provides a new inner bound for $K^*(S)$.
Let us briefly describe how to obtain it before giving its precise definition.   
By Corollary \ref{quasi-final} all points in $\T^2(S)$ belong to $K^*(S)$. Among these we select the uniaxial ones, namely the points of intersection 
of the curves \eqref{condnec1} and \eqref{condnec2} with the uniaxial lines $t_1=t_2$ and $t_2=t_3$ respectively. We denote such points by $U_{\alpha}$ and 
$U_{\beta}$ (see Definition \ref{def_U} and Figure \ref{figura2}).
We then define $\Gamma_{\alpha}$ and $\Gamma_{\beta}$ as the projection on $K^*(S)$ of the rank-one segments connecting a specific multiple of $S$ to 
$R_\alpha^t U_{\alpha}R_\alpha$ and $R_\beta^t U_{\beta}R_\beta$ respectively (Definition \ref{def33}).  
The set $\LL$ is finally defined as the set enclosed, in the unit trace plane, by  $\Gamma_{\alpha}$, $\Gamma_{\beta}$ and appropriately reflected and rotated copies of $\Gamma_{\alpha}$ and $\Gamma_{\beta}$ (see Definitions \ref{def:Gamma-construction}, \ref{matL} and  Figures~\ref{fig:sextant-annotated}-\ref{fig:hexagon-annotated}).
Throughout the present section, we assume that $S$ and $T$ satisfy \eqref{S}, \eqref{cS} and \eqref{ordering}. 
\begin{definition}\label{def_U}
Set \begin{equation}\label{upm}
\begin{array}{cc}
U_{\alpha}=
\left(
\begin{array}{ccc}
u_{\alpha}&0&0\\
0&u_{\alpha}&0\\
0&0&1-2 u_{\alpha}
\end{array}
\right), 
\quad U_{\beta}=
\left(
\begin{array}{ccc}
1-2 u_{\beta}&0&0\\
0&u_{\beta}&0\\
0&0&u_{\beta}
\end{array}
\right),
\end{array}
\end{equation}
where $u_{\alpha}$ and $u_{\beta}$ are the smallest and greatest roots of 
\begin{equation}\label{H}
H(x):=6 s_2\,x^2+ x\,(s_1 s_3-3 s_2-4 s_2^2)+ 2 s_2^2
\end{equation}
respectively.
Set
\begin{equation}\label{mpd1}
\begin{array}{ccc}
n_\alpha=\left(
\begin{array}{c}
\cos \varphi_\alpha\\0\\ \sin \varphi_\alpha\end{array}
\right),\,
R_\alpha=
\left(
\begin{array}{ccc}
0&-1&0\\
-\cos \theta_\alpha &0&\sin \theta_\alpha\\
-\sin \theta_\alpha &0&-\cos  \theta_\alpha
\end{array}
\right),\\\\
n_\beta=\left(
\begin{array}{c}
\cos \varphi_\beta\\0\\ \sin \varphi_\beta\end{array}
\right),\,
R_\beta=
\left(
\begin{array}{ccc}
\cos \theta_\beta&0&-\sin \theta_\beta\\
\sin \theta_\beta&0&\cos  \theta_\beta\\
0&-1&0
\end{array}
\right),
\end{array}
\end{equation}
%and using \eqref{z-} and \eqref{z+}, we find
with 
\begin{equation}\label{mpd2}
\begin{array}{l}
    \displaystyle{\cos (2\varphi_\alpha)=\frac{2s_2(3s_2-1)+u_\alpha(1+s_1^2-9s_2^2-2s_1s_3+s_3^2)}{2(s_3-s_1)(s_2-u_\alpha)}},\\\\
     \displaystyle{\cos (2\varphi_\beta)=\frac{2s_2(3s_2-1)+u_\beta(1+s_1^2-9s_2^2-2s_1s_3+s_3^2)}{2(s_3-s_1)(s_2-u_\beta)}},
\end{array}
\end{equation}
\begin{equation}    \label{mpd3}
\begin{array}{ll}
\displaystyle{\cos(2 \theta_\alpha) =\frac{u_\alpha(s_3-s_1)+(u_\alpha
-s_2) \cos(2 \varphi_\alpha)}{s_2(1-3 u_\alpha)},\, \cos(2 \theta_\beta) =
\frac{u_\beta(s_1-s_3)+(s_2-u_\beta
) \cos(2 \varphi_\beta)}{s_2(1-3 u_\beta)}.}
\end{array}
\end{equation}
\end{definition}
\begin{remark}\label{remunico}
Observe that, by construction, 
$$
\alpha_-(U_{\alpha},S)=
%\alpha_+(U_{\alpha},S) = 
\frac{u_{\alpha}}{s_2},\quad
%\beta_-(U_{\beta},S)= 
\beta_+(U_{\beta},S) = \frac{u_{\beta}}{s_2}.
$$
We also find, see Appendix \ref{AppB}, that
\begin{equation}\label{ref1}
\alpha_+(U_{\alpha},S) = \frac{u_{\alpha}}{s_2},\quad
\beta_-(U_{\beta},S)= \frac{u_{\beta}}{s_2}.
\end{equation}
\end{remark}
We then set 
\begin{equation}\label{alfa-beta}
\alpha:=  \frac{u_{\alpha}}{s_2},
\quad
\beta:=  \frac{u_{\beta}}{s_2}.
\end{equation}
The next proposition concerns two special elements of $\mathcal T^{2}(S)$, namely $U_{\alpha}$ and $U_{\beta}$ from Definition~\ref{def_U}.
\begin{proposition}\label{prop_opt_curve}
Let $U_\alpha, U_\beta, R_\alpha, R_\beta, n_\alpha, n_\beta$ be defined by \eqref{upm}-\eqref{mpd1}.
Then  
\begin{itemize}
\item[(i)]
$ s_1\leq   u_{\alpha}<\frac 1 3< u_{\beta}\leq  s_3$.\\
\item[(ii)]
The matrices $U_\alpha, U_\beta \in \mathcal T^{1}(S)$. Precisely, they are the unique solutions $T$ to \eqref{oldbp} for $\lambda = \alpha$ and $\lambda = \beta$  respectively, i.e., 
\begin{equation}
\label{2791}
 R_\alpha^t  U_{\alpha}R_\alpha = \alpha S +(1-\alpha)\, n_\alpha\otimes n_\alpha,
\end{equation}
\begin{equation}
\label{2792}
 R_\beta^t  U_{\beta}R_\beta = \beta S +(1-\beta)\, n_\beta\otimes n_\beta.
\end{equation}
\end{itemize}
\end{proposition}
\begin{proof}
(i) We have $$H(0)=2s_2^2>0, \quad H(1)=(3 - 2 s_2) s_2 + s_1 s_3>0, \quad H\left(\frac 1 3\right)=-\frac 1 3 (s_1 - s_2) (s_2 - s_3)<0.$$ Therefore, the two roots of $H(x)=0$ satisfy $u_{\alpha}< 1/3< u_{\beta}$.  The other two inequalities follow from \eqref{ordering}.
One checks the formula for $H(\frac 1 3)$ by writing 
$$-\frac 1 3 s_2+\frac 2 3 s_2^2+\frac 1 3 s_1 s_3 =
-\frac 1 3 s_2(s_1+s_2+s_3)+\frac 2 3 s_2^2+\frac 1 3s_1 s_3 
$$ 
and then factoring the expression.

\noindent
(ii) This follows from Remark \ref{remunico} and  Lemma \ref{LemmaF} in the cases $t_1=t_2$ or $t_2=t_3$ for which either $A_\alpha(T,S)$ or $A_\beta(T,S)$ degenerates to a single point ($\alpha$ or $\beta$ respectively).
\end{proof}

The matrices $U_\alpha$ and $U_\beta$  correspond to the blue and orange points in  Figure~\ref{figura2}.
In order to define $\Gamma_\alpha$ and $\Gamma_\beta$ we will need an efficient way to describe curves  in eigenvalue space. 
Consider the one-parameter family of matrices
\begin{align*}
& p\mapsto
p  R_\alpha^t  U_{\alpha}R_\alpha +(1-p) \alpha S,
 \quad p\in[0,1],\\
& p\mapsto
p  R_\beta^t  U_{\beta}R_\beta +(1-p) \beta S, 
 \quad p\in[0,1].
\end{align*}
Normalize to trace one and set
\begin{align}\label{s11}
& M_\alpha(p):=
\eta_\alpha(p) R_\alpha^t  U_{\alpha}R_\alpha +
(1- \eta_\alpha(p)) S, \quad 
\eta_\alpha(p):= \frac{p}{p + (1-p)\alpha},\\
\label{s1177}
& M_\beta(p):= 
\eta_\beta(p) R_\beta^t  U_{\beta}R_\beta +
(1- \eta_\beta(p)) S, \quad 
\eta_\beta(p):= \frac{p}{p + (1-p)\beta}.
\end{align}

\noindent
Denote by 
\begin{align}
\label{ordm1}
& m_1(\alpha,p)\leq m_2(\alpha,p)\leq m_3(\alpha,p), \\
\label{ordm2}
& m_1(\beta,p)\leq m_2(\beta,p)\leq m_3(\beta,p),
\end{align}
the eigenvalues of $M_\alpha(p)$ and $M_\beta(p)$, respectively.
\begin{definition}\label{def33}
The curves $\Gamma_\alpha, \Gamma_\beta$ are defined as follows. Consider the parametric curves associated to \eqref{ordm1}-\eqref{ordm2}:
   \begin{equation}\label{m1m2m3}
   p\mapsto m(\alpha,p)=(m_1(\alpha,p),m_2(\alpha,p),m_3(\alpha,p)),   \quad p\in[0,1],  \end{equation}
       \begin{equation}\label{m1m2m3p}
       p\mapsto m(\beta,p)=(m_1(\beta,p),m_2(\beta,p),m_3(\beta,p)), \quad p\in[0,1].
 \end{equation}
Then
   \begin{equation*}
\Gamma_{\alpha}:=m(\alpha,[0,1]),\quad 
\Gamma_{\beta}:=m(\beta,[0,1]).
  \end{equation*}
\end{definition}
The curves $\Gamma_\alpha, \Gamma_\beta$ are shown in Figure~\ref{fig:sextant-annotated}.
\begin{figure}\label{fig:curves-annotated}
\begin{subfigure}{0.33\textwidth}
\includegraphics[width=\textwidth]{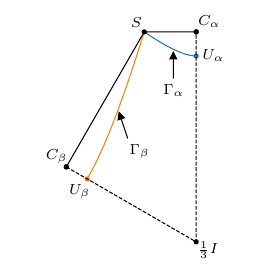}\caption{}
\label{fig:sextant-annotated}
\end{subfigure}%
\begin{subfigure}{0.33\textwidth}
\includegraphics[width=\textwidth]{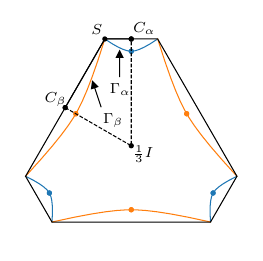}\caption{}\label{fig:hexagon-annotated}
\end{subfigure}%
\begin{subfigure}{0.33\textwidth}
\includegraphics[width=\textwidth]{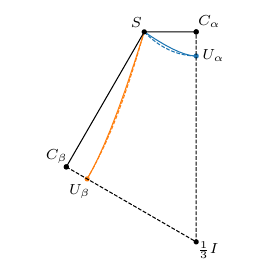}\caption{}\label{fig:sextant-comparison}
\end{subfigure}
\caption{
The left figure shows important fields in one sextant ($s_1\le s_2 \le s_3$) of the unit-trace plane. The outer quadrilateral connects the field $S$ to the isotropic field $\frac{1}{3}I$ and the two uni-axial points, $C_{\alpha}=(\frac{s_1+s_2}{2},\frac{s_1+s_2}{2},s_3)$ and $C_{\beta}=(s_1,\frac{s_2+s_3}{2},\frac{s_2+s_3}{2})$. The curves $\Gamma_{\alpha}$ and $\Gamma_{\beta}$ from Definition~\ref{def33} are also shown, together with their intersections, $U_\alpha$ and $U_\beta$ respectively, with the uniaxial lines. The center figure shows the construction in Definition~\ref{def:Gamma-construction}. The set $\LL$ is enclosed by the union of the reflected copies of $\Gamma_\alpha, \Gamma_\beta$. The right figure compares the curves from Figure~\ref{fig:sextant-annotated} to the set of $\mathcal{T}^{2}(S)$ fields shown in Figure~\ref{fig:T2-locations}.}
\label{figura2}
\end{figure}
\begin{definition}\label{def:Gamma-construction}
The closed curve $\Gamma$ is obtained as follows.
First, reflect $\Gamma_{\alpha}$ along the line $m_1=m_2$ in the plane $m_1+m_2+m_3=1$, then consider the union of the curves obtained with its $2\pi/3$ rotations within the unit trace plane.
Next, reflect $\Gamma_{\beta}$ along the line $m_2=m_3$ in the plane $m_1+m_2+m_3=1$, and consider the union of the curves obtained with its $2\pi/3$ rotations within the unit trace plane.
Finally, $\Gamma$ is the union of the six curves thus defined.  
\end{definition}
%%\noindent The construction of $\Gamma$ is shown in Figure~\ref{fig:hexagon-annotated}. By construction $\Gamma$ is a simple closed Jordan curve, and, for future reference, we note that  \hidden{occorre? forse va modificato perche' ho cambiato la def di $\eta$}
%\begin{equation}\label{m_2}
 %   m_2(\alpha,t) =s_2 \eta(\alpha,t),\quad m_2(\beta,t) =s_2 \eta(\beta,t),
%\end{equation}
% which readily follows from the fact that the second component of $n_\alpha$ and $n_\beta$ are zero.
\begin{definition}\label{matL}
    We denote by $\LL$ the bounded closed set enclosed by $\Gamma$ in the unit trace plane (see  Figure~\ref{fig:hexagon-annotated}).
\end{definition}

\begin{theorem}
 The set $\LL$ of Definition \ref{matL} satisfies 
 $\LL\subset K^*(S)$.
\end{theorem}

\begin{proof}
This is a consequence of Corollary \ref{quasi-final} and 
the fact that each point $P$ in the interior of $\LL$ 
lies on a segment that connects two points $P_1$ and $P_2$ on the boundary of $\LL$ and that is the projection on $K^*(S)$ 
of a rank-one segment in matrix space. The latter property is 
referred to as the straight line attainability property in 
\cite{NM}. Specifically, given an internal point $P$, take the line through $P$ and $(0,0,1)$ (or one of its permutations).
%namely, $t\in\R\to (t \theta, t(1-\theta), 1-t)$ with $\theta\in (0,1)$. 
Since $P$ is an internal point, one can find points of intersection of such line with the boundary of $\LL$, say $P_1$ and $P_2$, 
of which $P$ is a convex combination. Then it is easy to see that there exists $\lambda$ such that $P_1$ and $\lambda P_2$ are rank-one connected, which implies that the whole segment connecting $P_1$ and $P_2$ is contained in $K^*(S)$. 
\begin{comment}
the authors prove what they call the straight-line attainability property. In our context, this means the following. Suppose that $F^*$ and $G^*$ belong to $K^*(S)$ and 
that they both lie on a line passing through $(0,0,1)$ (or any of its cyclic permutations), 
%for some $s^*\in(0,1)$ they both lie on the line
%$t\to (t s^*, t(1-s^*), 1-t)$ (or any of the cyclic permutation of it)  and that they belong to $K^*(S)$, 
then the entire segment joining $F^*$ and $G^*$ belong to $K^*(S)$. Indeed one can find $\lambda\in\R$ such that $F^*$ and $\lambda G^*$ are rank-one connected along the direction $(0,0,1)\otimes (0,0,1)$.
\end{comment}
\end{proof}

\subsection{Comparison with previously known inner bound}
\label{vecchio}
In the present section we compare our results with those that were known prior to the present work.  
Nesi $\&$ Milton \cite{NM} established the existence of a non-trivial subset of $K^{*}(S)$. We denote  it by $\mathcal L_{MN}(S)$.  
The set $\mathcal L_{MN}(S)$ is a non-convex polygon. Like $\LL$, 
it is formed by six sets, each one obtained from another by an appropriate permutation of the eigenvalues. 
Therefore it suffices to define its restriction to one of the sextants formed by the uniaxial axes. We choose the upper left sextant, see Figure~\ref{figura3-c}.
\begin{definition}\label{defLMN}
The restriction of $\mathcal L_{MN}(S)$ to the sextant with one vertex at $S$ is the quadrilateral of vertices
$$
(s_1,s_2,s_3),\,V_\alpha,
\left(\frac 13,\frac 13,\frac 13\right),\,V_\beta
$$
with
\begin{equation*}
    V_\alpha=(v_\alpha,v_\alpha, 1-2v_\alpha),\,\,
    V_\beta=(1-2v_\beta, v_\beta,v_\beta),\,\,
    \displaystyle{v_\alpha=\frac{s_2}{2s_2+s_3}},\,\,
    \displaystyle{v_\beta=\frac{s_2}{2s_2+s_1}}.
\end{equation*}
\end{definition}

\begin{remark}
   The set $\mathcal L_{MN}(S)$ is the grey set in each image of Figures~\ref{figura3-a} and \ref{figura3-b}. 
   We check in Appendix \ref{AppB} that $v_\alpha$ and $v_\beta$ 
   satisfy 
\begin{equation}
\label{inequalities}
   \begin{array}{l}
\displaystyle{
H(v_\alpha)=-
   \frac{s_2 s_3 (s_3-s_2)(s_2-s_1)
   }{(2 s_2 + s_3)^2}
   }<0
   \\\\
 \displaystyle{
 H(v_\beta)=-
   \frac{s_1 s_2 (s_3-s_2)(s_2-s_1)}
   {(s_1 + 2 s_2)^2}
   }<0,
 \end{array}
\end{equation}
 where $H$ is defined by \eqref{H}, and that \eqref{inequalities} implies
 \begin{equation}\label{betteruni}
u_\alpha<v_\alpha<v_\beta<u_\beta .
 \end{equation}
\end{remark}
\begin{figure}
\centering
\begin{subfigure}{0.33\textwidth}
\includegraphics[width=\textwidth]{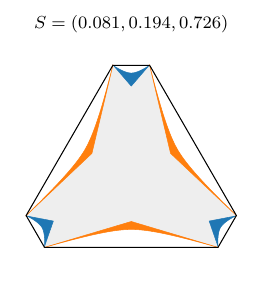}
\caption{}
\label{figura3-a}
\end{subfigure}%
\begin{subfigure}{0.33\textwidth}
\includegraphics[width=\textwidth]{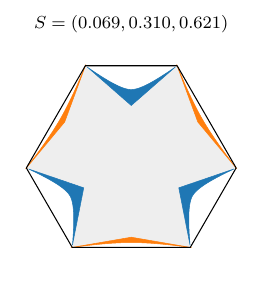}
\caption{}
\label{figura3-b}
\end{subfigure}%
\begin{subfigure}{0.33\textwidth}
\includegraphics[width=\textwidth]{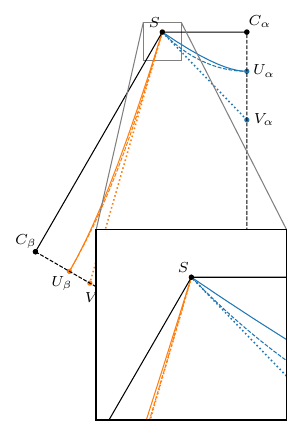}
\caption{}\label{figura3-c}
\end{subfigure}
\caption{The left two figures show comparisons between Figure~\ref{fig:hexagon-annotated}
with the constructions in~\cite{NM}. The shaded gray region is formed by the procedure described in~\cite{NM}. 
The blue and orange regions are additional fields found by the constructions in the present paper. They are new. The right figure shows the comparison between $\Gamma$ (solid lines), $\T^2(S)$ (dashed lines) and the boundary of $\mathcal L_{MN}(S)$ (dotted lines).
}
\label{fig:T2-envelope-with-NM-constructions}
\end{figure}

\begin{proposition}\label{MN_prop}
If \eqref{cS} holds, i.e., $S$ in not uniaxial, we have $\mathcal L_{MN}(S)\subset \mathcal L(S)$. 
\end{proposition}
\begin{proof}
The blue and orange triangle-like sets 
in  Figures \ref{figura3-a} and \ref{figura3-b}
are those where our construction does better than the previous one. 
Let us focus on the upper-left sextant. 
We will show that $\Gamma_\alpha$, which starts from $S$ and ends at $U_\alpha$, stays above the segment 
starting at $S$ and ending at $V_\alpha$. We denote the latter by $B_\alpha$ (the dotted line in Figure \ref{figura3-c}).
%and parametrize it (using $t=\eta(p)$), so that it is expressed in parametric form 
%by $t\to b_\alpha(t),\,t\in[0,1]$. 
\begin{comment}    
(the blue set in the upper left side of the exterior hexagon.  Its boundary is made of three curves: two segments and $\Gamma_\alpha$, see Figure \ref{fig:hexagon-annotated}. 
One segment is vertical and joins $U_\alpha$ and another uniaxial point, call it $V_\alpha$ .
\end{comment} 
  Taking into account \eqref{betteruni} and the convex bound, it is enough to prove that
  the curves $\Gamma_\alpha$ and $B_\alpha$ intersect only at the point $S$. 
  %$b_\alpha(t_1)\neq m(\alpha,t_2)$ for any $t_1,t_2\in(0,1)$. 
 Recalling \eqref{2791} and \eqref{s11}, assume on the contrary that for some $p_1,p_2\in(0,1)$ the matrices
\begin{equation}\label{1-3}
   A_1:= (1-p_1) S+ p_1 V_\alpha, \quad
    A_2:=(1-p_2)S+ p_2 (\alpha S+(1-\alpha)n_\alpha\otimes n_\alpha),
\end{equation}
share the same eigenvalues. 
Notice that the vector $e_2=(0,1,0)$ is a common eigenvector of the matrices $A_1$ and $A_2$, since $(n_\alpha)_2=0$. We start by imposing that the corresponding  eigenvalues  are  the same, which happens if and only if
\begin{equation*}
    (1-p_1) s_2+ p_1 v_\alpha=(1-p_2)s_2+ p_2 \alpha s_2,
\end{equation*}
for some $p_1,p_2\in(0,1)$,
  which gives
  \begin{equation}\label{t1t2}
    p_2= p_1\frac{s_2 (s_2-s_1 )}{(2 s_2 + s_3) (s_2 - u_\alpha)}=p_1v_\alpha \frac{s_2-s_1}{s_2 - u_\alpha}=p_1\frac{v_\alpha}{s_2}  \frac{s_2-s_1}{1- \alpha}.
\end{equation}
Now, consider the remaining $2\times 2$ block in \eqref{1-3}, obtained by removing the second row and column of the matrices $A_1$ and $A_2$. These are symmetric matrices. By definition, ${\rm Tr} A_1={\rm Tr} A_2=1$. Since by \eqref{t1t2} the eigenvalue corresponding to the 
 common eigenvector $e_2$ is the same for $A_1$ and $A_2$, the trace of the new $2\times 2$ matrices is also the same.
Setting
\begin{equation}\label{la}
L_\alpha:=1-p_2(1-\alpha),
\end{equation} 
we impose that their determinants are equal. This amounts to 
\begin{equation}\label{new1}
   ((1-p_1)s_1+ p_1 v_\alpha)((1-p_1)s_3+ p_1 (1-2v_\alpha))- 
   \Big(
   L_\alpha^2 s_1 s_3+L_\alpha(1-L_\alpha)\frac{\alpha^2 s_1 s_3-u_\alpha(1-2u_\alpha)}{\alpha(1-\alpha)} 
   \Big)= 0.
\end{equation} 
Using  \eqref{t1t2}, formula \eqref{la} can be written as
\begin{equation}\label{penultima}
 L_\alpha=\frac{f_1}{2s_2+s_3}, 
\end{equation}
where  $f_1 :=p_1(s_1-s_2) +(2s_2+s_3)$.
On the other hand, by \eqref{penultima} and the formula for $v_\alpha$ in Definition \ref{defLMN}, one has
\begin{equation}\label{Lf}
(1-p_1)s_3+ p_1 (1-2v_\alpha)= s_3 \frac{f_1}{2s_2+s_3}=s_3 L_\alpha.
\end{equation}
Thus \eqref{new1} is equivalent to
\begin{equation}\label{new2}
L_\alpha\Big(s_3  \big((1-p_1)s_1+ p_1 v_\alpha\big) -
L_\alpha s_1 s_3 - (1-L_\alpha)B\Big)=0,
\end{equation}
where 
\begin{equation}\label{BBB}
B:=\frac{\alpha^2 s_1 s_3-u_\alpha(1-2u_\alpha)}{\alpha(1-\alpha)}.
\end{equation}
We claim that \eqref{new2} is equivalent to
$f_2=0$, 
where
\begin{equation}\label{new4} 
f_2 := (2 s_2^2 + 2 s_1 s_3 + s_2 s_3 + s_3^2)\alpha - (s_2 + s_3). 
\end{equation}
Postponing to Appendix \ref{AppC} the proof of the above claim, 
we proceed and complete the proof.
Solving $f_2= 0$ in \eqref{new4} for $\alpha$, and recalling that 
$u_\alpha=\alpha s_2$, see \eqref{alfa-beta}, 
yields
$$
u_\alpha=\frac{s_2 (s_2 + s_3)}{2 s_2^2 + 2 s_1 s_3 + s_2 s_3 + s_3^2}.
$$
A simple calculation shows that 
$$u_\alpha-v_\alpha =\frac{2 s_2 s_3 (s_2 - s_1) }{(s_3+2s_2)(2 s_2^2 + 2 s_1 s_3 + s_2 s_3 + s_3^2)}>0,$$
which contradicts \eqref{betteruni}. 
The argument for $\Gamma_\beta$ is similar and is omitted.
\begin{comment}
One finds $t_2$ as a function of $t_1$; then the determinants of the $2\times2$ matrices, in this case, are equal if and only if either $t_1$ has a negative value, or
 $$
u_\beta=\frac{s_2 (s_2 + s_1) }{2 s_2^2 + 2 s_1 s_3 + s_2 s_1 + s_1^2}.
$$
In this case, we find
$$u_\beta-v_\beta =\frac{2 s_1 s_2 (s_2 - s_3) }{(s_1+2s_2)(2 s_2^2 + 2 s_1 s_3 + s_2 s_1 + s_1^2)}<0,$$
namely $u_\beta<v_\beta$, which contradicts \eqref{betteruni}. 
\end{comment}
\end{proof}

\begin{remark}\label{stability}
In fact one can prove that the set $\LL$ is stable under lamination, unlike the set enclosed by $\T^2(S)$ or the set $\mathcal L_{MN}(S)$. A precise stability theorem will be given in a forthcoming paper. 
\end{remark}
\begin{comment}
From the point of view of possible applications, it is always desirable to look for exact microgeometries as a concrete step toward manufacturing a composite. Given a uniaxial $S$, the unique isotropic matrix on the plane ${\rm Tr} S^*= 1$ is attained either by the Schulgasser microgeometry \cite{Sch}, i.e., an exact solution or via a very intricate approximate solution. By analogy, it is reasonable to pose the following question. 
Given a non-uniaxial $S$, does there exist a ``Schulgasser's type'' \cite{Sch}, i.e. an exact solution exhibiting $U_\pm$ as effective behavior? This would achieve a major simplification both from the practical and the theoretical point of view.
\end{comment}
%%%%%%%%%%%%%%%%%%%%%%%%%%%%%%%%%%%%%%%%%
%%%%%%%%%%%%%%%%%%%%%%%%%%%%%%%%%%%%%%%%%
%%%%%%%%%%%%%%%%%%%%%%%%%%%%%%%%%%%%%%%%%

\section{Proofs of  Lemma \ref{LemmaF}  and  Proposition \ref{prop5.2} }
\subsection{Proof of Lemma \ref{LemmaF}}\label{dimoLemma} 
We have that \eqref{oldbp} is verified if and only if the matrices $\lambda S+(1-\lambda) n\otimes n$ and  $T$ share the same eigenvalues, i.e.,  if and only if the following equivalence holds:
\begin{equation}
\det (\la S + (1-\la) n \otimes n - z I ) = 0 \Leftrightarrow z = t_j,\quad
j=1,2,3~.
\end{equation}
Preliminarly, we observe that, if \eqref{oldbp} holds then $\la>0$. Indeed, If $\la<0$, the matrix $\la S+(1-\la)n\otimes n$ must have at least two negative eigenvalues, contradicting the fact that $T$ is positive definite. Moreover $\la \neq 1$ (see Remark \ref{noninteressante}). Thus, without loss of generality, we shall assume that $\la\in(0,1)\cup(1,+\infty)$.
 
In addition, to begin with, we will assume that
$
t_j \neq \la s_i~, \,i,j=1,2,3
$,
i.e.,  
\begin{equation}\label{ns}
z\neq
\la s_i~,\, i=1,2,3.
\end{equation}
The remaining cases  follow by a continuity argument.
Thus,
\begin{equation}\label{cp}
\begin{array}{cc}
\det ( \la S + (1-\la) n \otimes n - z I ) = \det(\la S - z I) \det ( I +
(1-\la) (\la S - z I)^{-1} n\otimes n) 
\\\\
= \det(\la S - z I)\displaystyle{\left(1 +  (1-\lambda)\sum_{i=1}^{3} \frac{n_i^2}{\la s_i - z}\right)}~.
\end{array}\end{equation}
Assuming  \eqref{ns} and  using \eqref{cp} we see that 
  \begin{equation*}
\det (\la S + (1-\la) n \otimes n - z I ) = 0 \Leftrightarrow 1 +  (1-\la)\sum_{i=1}^{3} \frac{n_i^2}{\la
s_i - z} =0~.
\end{equation*}
 Allow $z$ to be complex and define 
\begin{equation}\label{ns2}
B(z) = 1 - (1 -\la) 
\sum_{i=1}^{3} \frac{n_i^2}{z - \la s_i}~.
\end{equation}
Clearly, $B(z)$ is the ratio of two monic polynomials of third degree, i.e., it has the
form
\begin{equation*} 
\frac{\prod_{i=1}^{3} (z -z_i)}{\prod_{i=1}^{3} (z -\la s_i)}.
\end{equation*}
Moreover, $B(z)$ has to vanish when $z = t_i~, i=1,2,3$. Thus, in fact, $B$
has the form
\begin{equation}\label{ns3}
B(z) = \frac{\prod_{i=1}^{3} (z -t_i)}{\prod_{i=1}^{3} (z -\la s_i)}.
% \displaystyle{\frac{\displaystyle{\prod_{i=1}^{3} (z -t_i)}}{\displaystyle{\prod_{i=1}^{3} (z -\la s_i)}}}.
\end{equation}
The strategy is to compute the residue of $B(z)$ at the points $z = \la s_i, i=1,2,3$
using the two expressions \eqref{ns2},  \eqref{ns3} for $B(z)$.  We then find that the 
matrices $\la S + (1-\la) n \otimes n$ and $T$ have
 the same eigenvalues provided the 
relations \eqref{ni}  are satisfied. 
From  \eqref{ns2} it follows that 
\begin{equation*}
{\rm Res\,} B(\lambda\,s_k)=\lim_{z\to \lambda\,s_k}(z- \lambda\,s_k)B(z)=
-(1-\lambda) n_k^2, \quad k=1,2,3.
\end{equation*}
On the other hand, using  \eqref{ns3} we find
\begin{equation*}
{\rm Res\,} B(\lambda\,s_k)=\lim_{z\to \lambda\,s_k}(z- \lambda\,s_k)B(z)
\end{equation*}
\begin{equation*}
= \lim_{z\to \lambda\,s_k} \frac{\prod_{i=1}^{3} (z -t_i)}{\prod_{i\neq k} (z -\la s_i)}=
\frac{\prod_{i=1}^{3} (\lambda\,s_k -t_i)}{\lambda^2 (s_k -s_p)(s_k -s_q)},\quad p\neq k\neq q. 
\end{equation*}
It follows
\begin{equation*}
n_k^2=-\frac{\prod_{i=1}^{3} (\lambda\,s_k -t_i)}{(1-\lambda)\lambda^2 (s_k -s_p)(s_k -s_q)}=
\frac{\prod_{i=1}^{3} (t_i-\lambda\,s_k )}{(1-\lambda)\lambda^2 (s_k -s_p)(s_k -s_q)}.
\end{equation*}
Thus, \eqref{ni} follows.
For given $t_i$ and $s_i$ the only free parameter is $\lambda\notin \{0,1\}$. A solution to our problem exists if and only if
 we may select $\lambda \in(0,1)\cup(1,+\infty)$ in such a way that
 the $n^2_i,~i=1,2,3$ are nonnegative (since the corresponding $n_i$ must be real)
and sum up to one. Indeed, $1={\rm Tr} (R^t TR) =\lambda {\rm Tr}S+ (1-\la)\sum n_i^2= \la +(1-\la)\sum n_i^2$, and $\la\ne 1$, so we have  $\sum n_i^2=1$.
Therefore, it suffices to find $\la$ such that 
\begin{equation}
\hbox{either $\lambda\in(0,1)$  and}
\begin{array}{ccc}
n_i^2\geq 0\Longleftrightarrow \
\left\{
\begin{array}{ccc}
a_1(\lambda):=(t_1 -\la s_1)(t_2 -\la s_1)(t_3 -\la s_1)\geq 0\\\\
a_2(\lambda):=(t_1 -\la s_2)(t_2 -\la s_2)(t_3 -\la s_2)\leq 0\\\\
a_3(\lambda):=(t_1 -\la s_3)(t_2 -\la s_3)(t_3 -\la s_3)\geq 0,
\end{array}
\right.
\end{array}
\end{equation}
 \begin{equation}
\hbox{or $\lambda>1$  and}
\begin{array}{ccc}
n_i^2\geq 0\Longleftrightarrow \
\left\{
\begin{array}{ccc}
a_1(\lambda):=(t_1 -\la s_1)(t_2 -\la s_1)(t_3 -\la s_1)\leq 0\\\\
a_2(\lambda):=(t_1 -\la s_2)(t_2 -\la s_2)(t_3 -\la s_2)\geq 0\\\\
a_3(\lambda):=(t_1 -\la s_3)(t_2 -\la s_3)(t_3 -\la s_3)\leq 0.
\end{array}
\right.
\end{array}
\end{equation}

Note that  the three roots $\frac{t_j}{s_k}>0$ of  each $a_i$ are positive.   We start treating the case $\lambda\in (0,1)$, the case $\la>1$ is similar and omitted.
We take a close look to the third order polynomials $a_i(\lambda)$ and observe that $a_i(0)>0, i=1,2,3$.
Thus,
\begin{equation}
\begin{array}{ccc}
a_1(\lambda)\geq0, \quad\lambda \in(0,1) &\Longleftrightarrow &\lambda\in\left( \left(-\infty,\frac{t_1}{s_1}\right]\bigcup  \left[\frac{t_2}{s_1},\frac{t_3}{s_1}\right]\right)\cap (0,1),\\\\
a_2(\lambda)\leq 0,\quad\lambda \in(0,1), &\Longleftrightarrow & \lambda\in\left( \left[\frac{t_1}{s_2},\frac{t_2}{s_2}\right]\bigcup  \left[\frac{t_3}{s_2},+\infty\right)\right)\cap (0,1),\\\\
a_3(\lambda)\geq0,\quad\lambda \in(0,1) &\Longleftrightarrow &\lambda\in\left( \left(-\infty,\frac{t_1}{s_3}\right]\bigcup  \left[\frac{t_2}{s_3},\frac{t_3}{s_3}\right]\right)\cap (0,1).\
\end{array}
\end{equation}
Since, $\frac{t_1}{s_1}\geq 1$, the first inequality is satisfied for any $\la\in(0,1)$.
\
The second inequality implies $\lambda\geq t_1/s_2$, and since $s_2<s_3$, we have $t_1/s_3<t_1/s_2$ and $t_3/s_2>t_3/s_3$. Hence, the second and third inequalities hold when
$$\lambda\in\left[\frac{t_2}{s_3},\frac{t_3}{s_3}\right]\cap\left[\frac{t_1}{s_2},\frac{t_2}{s_2}\right].$$
\begin{comment}If we had $\lambda<\frac{t_1}{s_3}$ then the {\bf second inequality} could not be satisfied, hence, the third must be satisfied when $\lambda \in\left[\frac{t_2}{s_3},\frac{t_3}{s_3}\right]\cap (-\infty,1]$. This implies
 $\lambda\leq \frac{t_3}{s_3}$ and therefore the second inequality is satisfied if and only if
 $\lambda\in \left[\frac{t_1}{s_2},\frac{t_2}{s_2}\right]$.
 [RIPRENDERE A CORREGGERE DA QUI]
Next, since $\lambda\geq\frac{t_1}{s_2}$ the first inequality holds when
$ \lambda\in  \left[\frac{t_2}{s_1},\frac{t_3}{s_1}\right]$.
\end{comment}
 Therefore, the set of admissible $\lambda\in(0,1)$ satisfy the condition
$$
\max \left(\frac{t_1}{s_2},\frac{t_2}{s_3}\right)\leq\lambda\leq \min\left(\frac{t_2}{s_2},\frac{t_3}{s_3}\right)=\min\left(\frac{t_1}{s_1},\frac{t_2}{s_2},\frac{t_3}{s_3}\right).
$$
\begin{comment}
    Indeed, $\lambda<1$ is implied by the right-hand side inequality since $\sum t_i=1$ and $\sum s_i=1$ implies that at least one of the ratios is less than or equal to one (equality holds if and only if $t_i=s_i$ for all $i$'s).
\end{comment}
\noindent
When $\lambda >1$, we get
$$
\max \left(\frac{t_1}{s_1},\frac{t_2}{s_2},\frac{t_3}{s_3}\right)=\max \left(\frac{t_1}{s_1},\frac{t_2}{s_2}\right)\leq\lambda\leq \min\left(\frac{t_2}{s_1},\frac{t_3}{s_2}\right).
$$
It remains to remove the assumption $t_j\neq\lambda \,s_k$. Indeed, when $t_j=\lambda \,s_k$ for some $j$ and $k$, the formulae \eqref{ni} for $n_i^2$ continue to hold with some $n_i=0$. Then, the problem effectively becomes two-dimensional and it is treated similarly.
Finally, \eqref{casoT=S} follows from the assumption that $S$ be not uniaxial implying  $\alpha_-(S,S)<1$.
%
%
\begin{comment}
\subsection{Proof of Proposition \ref{specialcase}}\label{forse}
Assume $A_{\alpha}(T,S)=\emptyset$. 
%\hidden{To ease notation, we drop the dependence upon $S$ and $T$ since they are fixed once and for all in the proof}.  
Then 
$\alpha_->\alpha_+$. In general, one has six cases for the pair $(\alpha_-,\alpha_+)$. The assumptions on $T$ imply
$$A_\alpha(T,S)=\left[\max\left(\frac{t_1}{s_2},\frac{t_2}{s_3}\right),
\min\left(\frac{t_2}{s_2},\frac{t_3}{s_3}\right)\right].$$ So we only have four cases. The first three are:
\begin{equation}
\begin{array}{l}
(\alpha_-,\alpha_+)=\left(\frac{t_1}{s_2},\frac{t_2}{s_2}\right),\quad\hbox{and $\alpha_->\alpha_+ \quad\Longrightarrow \quad t_1>t_2,\quad$ a contradiction.;}\\\\
(\alpha_-,\alpha_+)=\left(\frac{t_2}{s_3},\frac{t_2}{s_2}\right),\quad\hbox{and $\alpha_->\alpha_+ \quad\Longrightarrow \quad s_2>s_3,\quad$ a contradiction;}\\\\
(\alpha_-,\alpha_+)=\left(\frac{t_2}{s_3},\frac{t_3}{s_3}\right),\quad\hbox{and $\alpha_->\alpha_+ \quad\Longrightarrow \quad t_2>t_3,\quad$ a contradiction.}
\end{array}
\end{equation}
We are left with the only possibility
\begin{equation}
\begin{array}{l}
(\alpha_-,\alpha_+)=\left(\frac{t_1}{s_2},\frac{t_3}{s_3}\right),\quad\hbox{and $\alpha_->\alpha_+ \quad\Longrightarrow \quad \displaystyle{\frac{t_1}{s_2}>\frac{t_3}{s_3}}$}.
\end{array}
\end{equation}
Hence,
$$
A_{\alpha}(T,S)=\emptyset \Longrightarrow \frac{t_1}{t_3}>\frac{s_2}{s_3}.
$$
The proof for $A_\beta (T,S)$ is similar and is therefore omitted.
%
%
\end{comment}
%
\subsection{Proof of Proposition \ref{prop5.2}}\label{dimopropoT2}
We prove the case when \eqref{condnec1} holds. The case when \eqref{condnec2} holds follows along similar lines.
Let
\begin{equation}\label{um}
\begin{array}{lll}
n=\left(
\begin{array}{c}
\cos \varphi\\0\\\sin \varphi\end{array}
\right),&
R_1=
\left(
\begin{array}{ccc}
0&-1&0\\
-\cos \theta_1&0&\sin \theta_1\\
-\sin \theta_1&0&-\cos  \theta_1
\end{array}
\right),&
R_2=
\left(
\begin{array}{ccc}
\cos \theta_2&0&-\sin \theta_2\\
0&1&0\\
\sin \theta_2&0&\cos \theta_2
\end{array}
\right),
\end{array}
\end{equation}
for some suitably chosen angles $\varphi, \theta_1, \theta_2$. 
Then we have
\begin{equation}\label{as}
\lambda_1\,S+(1-\lambda_1)\,n\otimes n=
\left(
\begin{array}{ccc}
\lambda_1\,s_1+(1-\lambda_1) \cos^2  \varphi&0&(1-\lambda_1) \cos  \varphi\,\sin \varphi\\
0&\lambda_1\,s_2&0\\
(1-\lambda_1) \cos  \varphi\,\sin \varphi&0&\lambda_1\,s_3+(1-\lambda_1) \sin^2  \varphi
\end{array}
\right).
\end{equation}
On the other hand,
\begin{equation}\label{Q-}
 R_1^t T  R_1=
\left(
\begin{array}{ccc}
t_2\,\cos^2 \theta_1+t_3\,\sin^2 \theta_1&0&(t_3-t_2)\,\cos \theta_1\,\sin \theta_1\\
0&t_1&0\\
(t_3-t_2)\,\cos \theta_1\,\sin \theta_1&0&t_2\,\sin^2 \theta_1+t_3\,\cos^2 \theta_1
\end{array}
\right).
\end{equation}
Recalling that \(\lambda_1\,s_2=t_1\),  we get
\begin{equation}
%\label{first}
\lambda_1\,S+(1-\lambda_1)\,n\otimes n=R_1^t TR_1
\end{equation} if and only if
\begin{equation}\label{Rone}
\begin{array}{r}
\left(
\begin{array}{rrr}
t_2\,\cos^2 \theta_1+t_3\,\sin^2 \theta_1&(t_3-t_2)\,\cos \theta_1\,\sin \theta_1\\
(t_3-t_2)\,\cos \theta_1\,\sin \theta_1&t_2\,\sin^2 \theta_1+t_3\,\cos^2 \theta_1
\end{array}
\right)\\\\=
\left(
\begin{array}{ccc}
\lambda_1\,s_1+(1-\lambda_1) \cos^2  \varphi&(1-\lambda_1) \cos  \varphi\,\sin \varphi\\
(1-\lambda_1) \cos  \varphi\,\sin \varphi&\lambda_1\,s_3+(1-\lambda_1) \sin^2  \varphi
\end{array}
\right).\\\\
\end{array}
\end{equation}
The two matrices in \eqref{Rone} are symmetric and, by construction, their common trace equals $s_1+s_3$. Hence, they have the same eigenvalues if and only if their determinants are the same, i.e., if and only if
\begin{equation}\label{z-}
\lambda_1^2 s_1\,s_3+\lambda_1\,(1-\lambda_1)(s_3\,\cos^2 \varphi+ s_1\,\sin^2 \varphi)= t_2\,t_3.
\end{equation}
Setting 
\begin{equation}\label{zetaphi}
z(\varphi):=s_3\,\cos^2 \varphi+s_1\,\sin^2 \varphi,
\end{equation}
we write \eqref{z-} as
\begin{equation}\label{z1}
z(\varphi)=z_1(T,S):=\frac{ t_2\,t_3-\lambda_1^2\,s_1\,s_3}{\lambda_1\,(1-\lambda_1)}=\frac{ s_2^2\, t_2\, t_3-s_1 s_3\, t_1^2 }{(s_2 - t_1)\, t_1}.\end{equation}
The same calculation for the second index yields
\begin{equation}
R_2^t TR_2=
\begin{array}{lll}
\left(
\begin{array}{ccc}
t_1\,\cos^2 \theta_2+t_3\,\sin^2 \theta_2&0&(t_3-t_1)\,\cos \theta_2\,\sin \theta_2\\
0&t_2&0\\
(t_3-t_1)\,\cos \theta_2\,\sin \theta_2&0&t_3\,\cos^2 \theta_2+t_1\,\sin^2 \theta_2
\end{array}
\right),
\end{array}
\end{equation}
and recalling that \(\lambda_2\,s_2=t_2\),  we get
\begin{equation}
%\label{second}
\lambda_2\,S+(1-\lambda_2)\,n\otimes n=R_2^t TR_2
\end{equation} if and only if
\begin{equation*}%\label{Rtwo}
\begin{array}{r}
\left(
\begin{array}{ccc}
t_1\,\cos^2 \theta_2+t_3\,\sin^2 \theta_2&(t_3-t_1)\,\cos \theta_2\,\sin \theta_2\\
(t_3-t_1)\,\cos \theta_2\,\sin \theta_2&t_1\,\sin^2 \theta_2+t_3\,\cos^2 \theta_2
\end{array}
\right)\\\\=
\left(
\begin{array}{ccc}
\lambda_2\,s_1+(1-\lambda_2) \cos^2  \varphi&(1-\lambda_2) \cos  \varphi\,\sin \varphi\\
(1-\lambda_2) \cos  \varphi\,\sin \varphi&\lambda_2\,s_3+(1-\lambda_2) \sin^2  \varphi
\end{array}
\right).
\end{array}
\end{equation*}
As in the previous calculation, we deduce that they share the same eigenvalues if and only if 
\begin{equation*}
\lambda_2^2 s_1\,s_3+\lambda_2\,(1-\lambda_2)(s_3\,\cos^2 \varphi+ s_1\,\sin^2 \varphi)= t_1\,t_3\,.
\end{equation*}
Recalling \eqref{zetaphi}, the previous equation requires
\begin{equation}\label{z2}
z(\varphi)=z_2(T,S):=\frac{ t_1\,t_3-\lambda_2^2\,s_1\,s_3}{\lambda_2\,(1-\lambda_2)}=\frac{ s_2^2\, t_1\, t_3-s_1 s_3\, t_2^2 }{(s_2 - t_2)\, t_2}.\end{equation}
The pair $(T,S)$ satisfies 
\begin{equation*}
z_1(T,S)=z_2(T,S)
\end{equation*}
if and only if
\begin{equation}\label{condition1}
\frac{ s_2^2\, t_2\, t_3-s_1 s_3\, t_1^2 }{(s_2 - t_1)\, t_1}=
\frac{s_2^2\,t_1\,t_3- s_1\,s_3\,t_2^2}{t_2(s_2-t_2)}.  \end{equation}
By \eqref{zetaphi}, when \eqref{condition1} holds,  the common value $z(\varphi)$ takes exactly any value belonging to $ [s_1,s_3]$.
Moreover, if $t_1\neq t_2$, then \eqref{condition1} is equivalent to the second condition in \eqref{condnec1}. we conclude that if the pair $(T, S)$ satisfies \eqref{condnec1}, then \eqref{main} holds.

\appendix
\section{Convex integration tools}\label{CT}
\noindent
We denote by $\measures$ the set of signed Radon measures on $\mathbb{M}^{3\times 3}$ having finite
mass. 
%By Riesz's representation theorem, we can identify $\measures$ with the dual of the space $C_0 (\Rmn)$. 
Given $\nu \in \measures$ we define its 
\textit{barycenter} as
\[
\bar{\nu} := \int_{\matrici} A \, d\nu(A) \,.
\]  
If $\Omega\subset\R^3$ is a bounded open domain, 
we say that a map $f \in C(\bar{\Omega}; \R^3)$ is \textit{piecewise affine} if there exists a countable family of pairwise disjoint open subsets $\Omega_i \subset \Omega$ with $\va{\partial \Omega_i}=0$ and 
\[
\va{ \Omega \smallsetminus \bigcup_{i=1}^{\infty} \Omega_i  } =0  \,,
\]
such that $f$ is affine on each $\Omega_i$. Two matrices $A , B \in \matrici$
such that $\rank (B-A)=1$ are said to be \textit{rank-one connected} and the measure
$\lambda \delta_A + (1- \lambda) \delta_B \in \measures$ with $\lambda \in [0,1]$ is 
called a \textit{laminate of first order} (see also \cite{Mnotes}, \cite{MS}, \cite{pedregal}).

\begin{definition}\label{laminates}
The family of \textit{laminates of finite order} $\laminates$ is the smallest family of 
probability measures in $\measures$	 satisfying the following conditions: 
\begin{enumerate}
\item[(i)] 	$\delta_A \in \laminates$ for every $A \in \matrici \,$;
\item[(ii)] 	 assume that $\sum_{i=1}^N \lambda_i \delta_{A_i} \in \laminates$ and $A_1=\lambda B + (1-\lambda)C$ with $\lambda \in [0,1]$ and $\rank(B-C)=1$. Then the probability measure
\[
\lambda_1 (\lambda \delta_B + (1-\lambda ) \delta_C) + \sum_{i=2}^N \lambda_i \delta_{A_i} 
\]
is also contained in $\laminates$.
\end{enumerate}
\end{definition}
\noindent
The process of obtaining new measures via (ii) is called \textit{splitting}.
The following proposition provides a fundamental tool to solve differential inclusions using convex integration (see e.g. \cite[Proposition 2.3]{afs} for a proof).
\begin{proposition} \label{gradienti}
Let $\nu = \sum_{i=1}^N \alpha_i \delta_{A_i} \in \laminates$ be a laminate of finite order with barycenter $\bar\nu=A$, that is $A= \sum_{i=1}^N \alpha_i A_i$ with $\sum_{i=1}^N \alpha_i=1$. Let $\Omega \subset \R^3$ be a bounded open set, $\alpha\in (0,1)$ and $0<\delta < \min \va{A_i-A_j}/2$. Then there exists a piecewise affine Lipschitz map 
$u \colon \Omega \mapsto \R^3$ such that
\begin{enumerate}
%[\indent (i)]
\item[(i)] $u(x)=Ax$ on $\partial \Omega$, \smallskip
\item[(ii)]  $\holder{u-A} < \delta$ , \smallskip
\item[(iii)]  $\va{ \{ x \in \Omega \, \colon \, \va{\nabla u (x) - A_i} < \delta \}  } = \alpha_i \va{\Omega}$, \smallskip
\item[(iv)]  $\dist$ $(\nabla u(x), \spt \nu) < \delta\,$ a.e. in $\Omega$. 	\smallskip
\end{enumerate} 
Moreover, if $A_i \in \R^{3\times 3}_{\rm sym}$, then the map $u$ can be chosen so that $u = \nabla f$, for some $f\in W^{2,\infty}(\Omega)$.
\end{proposition}

\section{Proof of \eqref{ref1} and \eqref{inequalities}}\label{AppB}
We start proving the first formula of \eqref{ref1}. The proof of the second is similar and omitted. 
To prove that $\alpha_+(S,U_\alpha(S))=\frac{u_\alpha}{s_2}$, 
we need to verify that
$\min\left(\frac{1-2u_\alpha}{s_3},\frac{u_\alpha}{s_2}\right)=\frac{u_\alpha}{s_2}$. The latter holds if and only if 
$$\frac{u_\alpha}{s_2}\leq\frac{1-2u_\alpha}{s_3}\Longleftrightarrow
u_\alpha\leq\frac{s_2}{2 s_2+s_3}\Longleftrightarrow H\left(\frac{s_2}{2 s_2+s_3}\right)<0\Longleftrightarrow H\left(v_\alpha\right)<0,$$ 
with $v_\alpha$ as in Definition \ref {defLMN} and  $H$  defined by \eqref{H}. This shows that the first formula of \eqref{ref1} holds if and only if the first formula of \eqref{inequalities} holds.
We claim that
\begin{equation}\label{1}
H\left(v_\alpha\right)=\frac{s_2 s_3(s_2-s_1)(s_2-s_3)}{(2 s_2+s_3)^2}.
\end{equation}
Assuming \eqref{1}, and using the ordering
 $s_1<s_2<s_3$, we  conclude $H\left(v_\alpha\right)<0$.
Let us prove \eqref{1}. 
\begin{align*}
H\left(v_\alpha\right)&=6s_2 \left(\frac{s_2}{2 s_2+s_3}\right)^2+(s_1 s_3-3s_2-4 s_2^2)\left(\frac{s_2}{2 s_2+s_3}\right)+2s_2^2\\\\&=
\frac{s_2}{(2 s_2+s_3)^2}[6 s_2^2 +(2 s_2+s_3)(s_1 s_3-3s_2-4 s_2^2)+2s_2(2 s_2+s_3)^2]\\\\&=
\frac{s_2}{(2 s_2+s_3)^2}[6 s_2^2 +(2 s_2+s_3)(s_1 s_3-3s_2-4 s_2^2+2s_2(2 s_2+s_3))]\\\\&=
\frac{s_2}{(2 s_2+s_3)^2}[6 s_2^2 +(2 s_2+s_3)(s_1 s_3-3s_2+2s_2s_3)]\\\\&=
\frac{s_2}{(2 s_2+s_3)^2}[6 s_2^2-6 s_2^2-3s_2 s_3+(2 s_2+s_3)(s_1 s_3+2s_2s_3)\\\\
\end{align*}
\begin{align*}
&=\frac{s_2 s_3}{(2 s_2+s_3)^2}[-3s_2 +(2 s_2+s_3)(s_1+2s_2)]\\\\ 
&=\frac{s_2 s_3}{(2 s_2+s_3)^2}[s_2(-3+2s_1+4 s_2) +s_3(s_1+2s_2)]\\\\
&=\frac{s_2 s_3}{(2 s_2+s_3)^2}[s_2(-3(s_1+s_2+s_3)+2s_1+4 s_2) +s_3(s_1+2s_2)]
\\\\
&=\frac{s_2 s_3}{(2 s_2+s_3)^2}[-s_1s_2+s_2^2-3 s_2 s_3+s_1 s_3+2 s_2 s_3]\\\\
&=\frac{s_2 s_3}{(2 s_2+s_3)^2}[s_2(s_2-s_1)+s_3(s_1-s_2)]=
\frac{s_2 s_3}{(2 s_2+s_3)^2}(s_2-s_1)(s_2-s_3),
\end{align*}
where we used the identity $s_1+s_2+s_3=1$ in the eight equality.

In order to prove the second formula in \eqref{inequalities}, note that  $H$  satisfies $H(s_1,s_2,s_3)=H(s_3,s_2,s_1)$.
Hence, using  \eqref{1},  we have
\begin{align*}H(v_\beta)= H\left(\frac{s_2}{2 s_2+s_1}\right)=H\left(\frac{s_2}{2 s_2+s_3}\right)\Big|_{s_3\to s_1}=
\frac{s_2 s_1(s_2-s_3)(s_2-s_1)}{(2 s_2+s_1)^2}\\\\
=-\frac{s_1s_2 (s_3-s_2)(s_2-s_1)}{(2 s_2+s_1)^2}<0.
\end{align*}

\section{Proof of the equivalence of \eqref{new2} and $f_2 =0$} \label{AppC}
We need to prove that 
$$
L_\alpha\Big(s_3  \big((1-t_1)s_1+ t_1 v_\alpha\big) -
L_\alpha s_1 s_3 - (1-L_\alpha)B\Big)=0
\Longleftrightarrow f_2=0,
$$
where $f_2$ is defined by \eqref{new4}.
Note that, by \eqref{Lf}, $L_\alpha=0$ if and only if $t_1=\frac{2s_2+s_3}{s_2-s_1}>1$, a contradiction.
Therefore we will prove that 
$$
s_3  \big((1-t_1)s_1+ t_1 v_\alpha\big) -
L_\alpha s_1 s_3 - (1-L_\alpha)B=0
\Longleftrightarrow f_2=0.
$$
Let us compute
\begin{align*}
& s_3  \big((1-t_1)s_1+ t_1 v_\alpha\big) -
L_\alpha s_1 s_3 - (1-L_\alpha)B \\\\
& =s_1 s_3+s_3t_1(v_\alpha-s_1)-
[L_\alpha s_1 s_3+(1-L_\alpha)B]\\\\
& =
s_1 s_3+t_1 s_3(v_\alpha-s_1)-\left[B+(s_1s_3-B)\left(1-t_1\frac{s_2-s_1}{2s_2+s_3}\right)\right]\\\\
& =
s_1 s_3+t_1s_3(v_\alpha-s_1)-s_1 s_3 +(s_1 s_3-B)t_1\left(\frac{s_2-s_1}{2s_2+s_3}\right)
\end{align*}
\begin{align*}
& =t_1\left(s_3(v_\alpha-s_1)+(s_1s_3-B)\left(\frac{s_2-s_1}{2s_2+s_3}\right) \right)\\\\
&=t_1 \left[s_3\frac{s_2 (s_1+s_2+s_3)-s_1(2 s_2+s_3)}{2 s_2+s_3}+
\frac{(s_1s_3-B)(s_2-s_1)}{2s_2+s_3}\right]\\\\
&=
 t_1 \frac{s_2-s_1}{2 s_2+s_3}\big( s_3(s_3+s_2)+
s_1 s_3-B\big)= t_1 \frac{s_2-s_1}{2 s_2+s_3}\left( s_3(s_3+s_2+s_1)-
B\right)\\\\
&=
t_1 \frac{s_2-s_1}{2 s_2+s_3}\big( s_3-
B\big).
\end{align*}
Recall that $t_1\in(0,1)$ and, by definition, $s_2\neq s_1$. Hence 
we are left to show that
 \begin{align*}
 s_3-B=0 \Longleftrightarrow f_2=0.
\end{align*}
Recalling \eqref{BBB}, by \eqref{alfa-beta}, we can write
\begin{equation}\label{AppC_B}
B=\frac{\alpha(s_1s_3+2s_2^2)-s_2}{1-\alpha}=\frac{\alpha(s_1s_3+2s_2^2)-s_2(s_1+s_2+s_3)}{1-\alpha}.\end{equation}
We now solve the equation $s_3-B=0$ with respect to the variable $\alpha$ using \eqref{AppC_B} and the identity $s_1+s_2+s_3 = 1$. 
Then we have that  $s_3-B=0$  if and only if 
\begin{align}\label{B}
\alpha=\frac{ s_2 + s_3}{
s_1 s_3+2s_2^2+s_3(s_1+s_2+s_3) }
=\frac{s_2 + s_3}{
 2 s_2^2 + 2 s_1 s_3 + s_2 s_3 + s_3^2}.
\end{align}
By \eqref{new4}, \eqref{B} is the same as $f_2=0$.
\section*{Acknowledgments} This material is based upon work supported by the National Science Foundation under Grant No.~2108588 (N.~Albin, Co-PI). V. Nesi gratefully acknowledges  Progetti di Ricerca di Ateneo 2022, ``Equazioni ellittiche e paraboliche non lineari'', rif. RM1221816BBB81FA.
M. Palombaro is a member of the Gruppo Nazionale per l'Analisi Matematica, la Probabilit\`a e le loro applicazioni (GNAMPA) of the Istituto Nazionale di Alta Matematica (INdAM). 

%%%%%%%%%%%%%%%%%
\section*{Declarations}
The authors have no conflicts of interest. Data sharing is not applicable to this article as no datasets were generated or analysed during the current study.
%%%%%%%%%%%%%%%%%%%%%%%%%%%%%%%%%%%%%%%%%%%%%%%%%%%%%%%%%%%%%%%%%%%%

\end{document}